\pdfoutput=1
\documentclass[11pt,letterpaper]{article}

\usepackage[english]{babel}
\usepackage[T1]{fontenc}
\usepackage{stmaryrd}
\usepackage{algorithmic}
\usepackage{amssymb}
\usepackage{marginnote}
\usepackage[ruled,vlined,linesnumbered]{algorithm2e}
\usepackage{booktabs,tabularx,threeparttable}
\usepackage[normalem]{ulem}
\usepackage{tikz}
\usetikzlibrary{fit,backgrounds,positioning,shapes.geometric}
\usepackage{fullpage}
\usepackage{mathrsfs}
\usepackage{amsmath}
\usepackage{bbm}
\usepackage{graphicx}
\usepackage{stackengine}
\usepackage{amsthm}
\usepackage{subcaption}
\usepackage{authblk}
\usepackage{thmtools}
\usepackage{thm-restate}
\usepackage[disable]{todonotes}
\usepackage{fancyhdr,lastpage}
\usepackage{amsfonts}
\usepackage[dvipsnames]{xcolor}
\usepackage{enumitem}
\usepackage[colorlinks=true,allcolors=blue]{hyperref}
\usetikzlibrary{calc, intersections, positioning}

\makeatletter
\renewcommand{\maketitle}{%
\begin{center}
{\large\bfseries \@title\par}
\vspace{0.5em}
{\normalsize \@author\par}
\end{center}
}
\renewenvironment{abstract}{%
\par\normalfont\normalsize
\begin{center}\begin{minipage}{0.9\linewidth}\small\noindent\textup{Abstract. }\ignorespaces}{\end{minipage}\end{center}\par}

\def\@seccntformat#1{\csname the#1\endcsname.\ }

\renewcommand\section{\@startsection{section}{1}{\z@}%
{2.0ex plus .5ex minus .2ex}{1.0ex plus .2ex}%
{\normalfont\large\centering}}

\renewcommand\subsection{\@startsection{subsection}{2}{\z@}%
{1.75ex plus .5ex minus .2ex}{-1em}%
{\normalfont\normalsize\bfseries}}
\makeatother

\newtheorem{thm}{Theorem}[section]

\newtheorem{lm}[thm]{Lemma}

\newtheorem{prop}[thm]{Proposition}

\theoremstyle{remark}

\theoremstyle{plain}

\newcommand{\set}[1]{\left\{#1\right\}}

\newcommand{\abs}[1]{\left| #1 \right|}

\newcommand{\N}{\mathbb{N}}

\begin{document}

\title{IMPROVED FRACTAL WEYL BOUNDS MATCHING IMPROVED SPECTRAL GAPS FOR HYPERBOLIC SURFACES AND OPEN QUANTUM MAPS}
\author{TRAVIS CUNNINGHAM}
\date{}
\maketitle

\begin{abstract}
 We prove a new fractal Weyl upper bound for the high-energy distribution of resonances of convex co-compact hyperbolic surfaces which matches the improved spectral gap given by Fourier decay. More precisely, if $\delta$ is the dimension of the limit set for the surface, $\beta_{\mathrm{BD}}>\tfrac{1}{2}-\delta$ is the fractal uncertainty exponent of Bourgain–Dyatlov, and
\[
m(\nu,\delta):=\min \Bigl(4\bigl(\nu-\beta_{\mathrm{BD}}\bigr),\,2\bigl(\nu-(\tfrac{1}{2}-\delta)\bigr),\,\delta\Bigr),
\]
then the number of resonances in the box $[R,R+1]+i[-\nu,0]$ is $O \bigl(R^{m(\nu,\delta)+\varepsilon}\bigr)$ as $R\to\infty$. This improves upon the fractal Weyl bound of Dyatlov, which matches the Patterson–Sullivan spectral gap $\tfrac{1}{2}-\delta$. We also give a new resolvent estimate improving the ones given by Dyatlov-Zahl and Dyatlov. Analogous results are obtained for quantum open baker’s maps, improving an estimate of Dyatlov-Jin, where we also give an improved fractal Weyl bound matching a spectral gap given by additive energy estimates. We refine known methods for proving fractal Weyl bounds which reduce the problem to an estimate of a certain determinant function; however, we use a different determinant function which allows us to make sharper estimates by applying the method of proof of the fractal uncertainty principle in each setting.
\end{abstract}

\section{INTRODUCTION}

\subsection{Main results.} Let $M=\Gamma\backslash\mathbb{H}$ be an infinite-area, convex co-compact hyperbolic surface; that is, $M$ is the quotient of the hyperbolic space $\mathbb{H}$ by a geometrically finite Fuchsian group $\Gamma$, containing 
no parabolic elements. We refer to \cite{Bor16} for a broad introduction to the spectral theory for such objects, which serve as models for more complicated systems in mathematical physics, and also have applications to algebra and number theory (see \cite{BGS11}).

This paper gives an improved fractal Weyl upper bound on the number of resonances - i.e., on the number of poles of the meromorphic continuation of the resolvent, $R(\lambda):=\left(-\Delta_{M}-\frac{1}{4}-\lambda^{2}\right)^{-1}$ - which \textit{matches} an improved spectral gap given by Bourgain–Dyatlov, \cite{BD17}. By a spectral gap we mean a region of the form $\set{\operatorname{Im} \lambda \geq -\beta}$ for some $\beta > 0$, which has at most finitely many resonances.

To state the result, let $\delta$ denote the Hausdorff dimension of the limit set, $\Lambda_{\Gamma}$, defined as the set of limit points of all orbits, $\Gamma z$, for $z\in\mathbb{H}$. A classical result of Patterson and Sullivan \cite{Pat76,Sul79} states that there is a simple resonance at $\lambda=-i \left(\tfrac{1}{2}-\delta\right)$ and no other resonance in the region $\{\operatorname{Im}\lambda\ge -(\tfrac{1}{2}-\delta)\}$. Naud \cite{Nau05} improved this result by showing that for some $\beta>\tfrac{1}{2}-\delta$ which depends on the surface, there are at most finitely many resonances in $\{\operatorname{Im}\lambda\ge -\beta\}$. Using the powerful new tool of the fractal uncertainty principle (see \cite{Dya19a} for an introduction), Bourgain–Dyatlov \cite{BD17} further improved this spectral gap of Naud by showing that the result holds for some $\beta_{BD}>\tfrac{1}{2}-\delta$ depending only upon $\delta$ and not on the specific features of the surface.

Let $R,\nu>0$, and
\begin{equation*}
\mathcal{N}(R,\nu):=\#\{\lambda\ \text{resonance}:\ \operatorname{Re}\lambda\in[R,R+1],\ \operatorname{Im}\lambda\ge -\nu\}, \tag{1.1} \label{eq:101}
\end{equation*}
where we always count resonances with multiplicities. Then we shall prove the following improved fractal Weyl upper bound on the density of resonances:

\newpage 

\begin{thm} \label{thm:1.1}
Let $M$ be a convex co-compact hyperbolic surface with $\delta\in(0,1)$, and let $\beta_{BD}$ be the size of the spectral gap proved in \cite[Theorem~1]{BD17}. Then for each $\nu>0$, $\varepsilon>0$, there exists $C>0$ such that
\[
\begin{gathered}
\mathcal N(R,\nu)\leq C R^{m(\nu,\delta)+\varepsilon},\quad R\to\infty,\\
m(\nu,\delta):=\min\left(4\left(\nu-\beta_{BD}\right),\,2\left(\nu-\left(\tfrac{1}{2}-\delta\right)\right),\,\delta\right).
\end{gathered}
\]
\end{thm}

This bound matches the spectral gap of \cite{BD17}, since the exponent $m(\nu,\delta)$ is negative for $\nu<\beta_{BD}$. Since $\beta_{BD}>\tfrac{1}{2}-\delta$ always holds, Theorem~\ref{thm:1.1} improves a similar estimate of \cite{Dya19b} which matches the Patterson–Sullivan gap, see Figure 1. We mention that \cite{BD18} gave a fractal uncertainty principle, and hence spectral gap, with $\beta>0$, even when $\tfrac{1}{2}-\delta<0$. This implies the bound $\mathcal N(R,\nu)=O(1)$ as $R\to\infty$ when $\nu<\beta$. Thus, although Theorem~\ref{thm:1.1} always beats the bound of \cite[Theorem~1]{Dya19b}, it may or may not give new information when $\tfrac{1}{2}-\delta<0$. However, Theorem~\ref{thm:1.1} always improves known estimates on resonance density when $\tfrac{1}{2}-\delta\ge 0$.

We also prove the following improved resolvent upper bound inside the Bourgain–Dyatlov spectral gap:
\begin{thm}\label{thm:1.2}
Assume that $\beta_{BD}>0$. Then for each $\nu\in(0,\beta_{BD})$, $\psi\in C_c^\infty(M)$, there exists $C_0>0$ such that for any $\varepsilon>0$,
\[
\|\psi R(\lambda)\psi\|_{L^{2}\to L^{2}} \leqslant C|\lambda|^{-1+c(\nu,\delta)+\varepsilon},\quad \operatorname{Re}\lambda \geqslant C_0,\ \operatorname{Im}\lambda \in [-\nu,1], \tag{1.2} \label{eq:102}
\]
where
\[
c(\nu,\delta)=\left(1+\frac{1-\delta-2\beta_{BD}}{1-\delta-2\nu}\right)\nu.
\]
\end{thm} 

We note that $\beta_{BD}>0$ always holds when $\delta\leq \tfrac{1}{2}$, but may also hold for some $\delta>\tfrac{1}{2}$. Combining \cite[Theorem~3]{DZ16} with the fractal uncertainty principle of \cite{BD17} implies the estimate (1.2) with $c$ replaced by $2\nu$. Since our exponent $c(\nu,\delta)<2\nu$ in the full range $\nu\in(0,\beta_{BD})$, we improve the bound given by \cite{DZ16}, \cite{BD17}. Similarly to Theorem~\ref{thm:1.1}, this theorem also improves a result of \cite{Dya19b} for the case of hyperbolic surfaces. Indeed, applied to surfaces, \cite[Theorem~2]{Dya19b} gives the bound (1.2) in the range $\nu\in\left(0,\tfrac{1}{2}-\delta\right)$ with $c$ replaced by $\nu(1-2\nu)/(1-\delta-2\nu)$. In addition to providing a bound on the larger range $(0,\beta_{BD})$, since our exponent $c(\nu,\delta)<\nu(1-2\nu)/(1-\delta-2\nu)$ on $\left(0,\tfrac{1}{2}-\delta\right)$, our result also improves the bound of \cite[Theorem~2]{Dya19b} on its range of validity. Thus, Theorem~\ref{thm:1.2} improves all known resolvent bounds for convex co-compact hyperbolic surfaces in the entire range $\left(0,\beta_{BD}\right)$, see Figure 1. Finally, we note that \cite{DW16} gives a lower resolvent bound which also applies to more general settings.

Next we turn to the related setting of quantum open baker's maps which are important models in the study of open quantum chaos, see \cite[§1.4]{DJ17} for more on that connection. Given a \emph{base} $M\in\mathbb{N}$, an \emph{alphabet} $\mathcal A\subset\mathbb{Z}_{M}=\{0,1,\ldots,M-1\}$, and a \emph{cutoff} $\chi \in C_{c}^{\infty}((0,1);[0,1])$, the corresponding quantum open baker's map is defined as an operator on $\ell_{N}^{2}:=\ell^{2}(\mathbb{Z}_{N})$ given by
\[
B_{N}:= \mathcal{F}_{N}^{*}\begin{pmatrix}
 \chi_{N/M}\, \mathcal{F}_{N/M}\, \chi_{N/M} & & & \\
& \ddots & & \\
& &  \chi_{N/M}\, \mathcal{F}_{N/M}\, \chi_{N/M}
\end{pmatrix} I_{\mathcal A,M}.
\]

\noindent Here $N=M^{k}$ for $k\in\mathbb{N}$, $ \mathcal{F}_{N}$ is the discrete Fourier transform, $ \chi_{N/M}$ is a discretization of $\chi$, and $I_{\mathcal A,M}$ is the diagonal matrix whose $j$th entry is $1$ if $\bigl\lfloor \tfrac{Mj}{N}\bigr\rfloor\in \mathcal A$ and $0$ otherwise. The set of trapped orbits for these systems (which corresponds to the limit set $\Lambda_{\Gamma}$ for hyperbolic surfaces) form Cantor sets, whose dimension is $\delta:=\frac{\log|\mathcal A|}{\log M}\in(0,1)$. See Section 4 for more details.

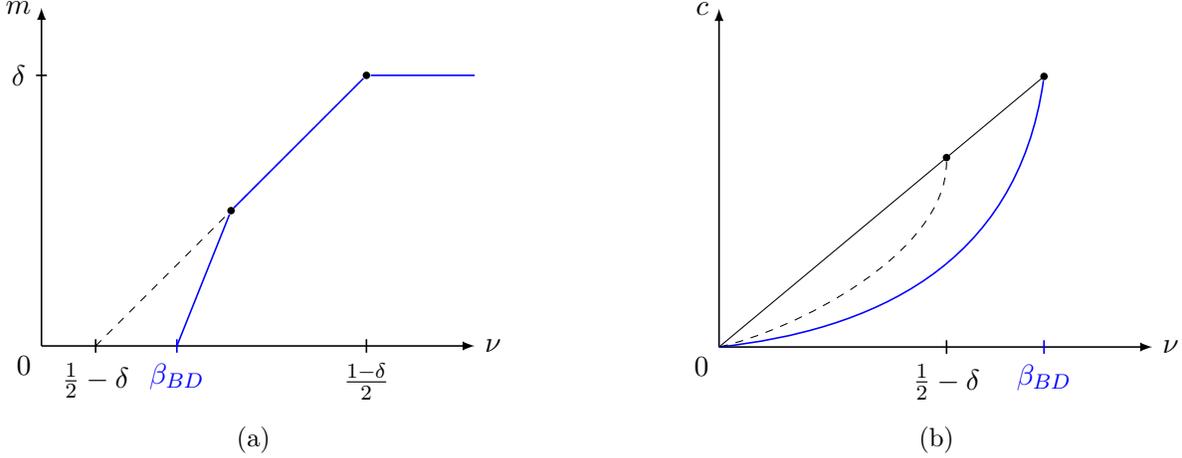
\begin{figure}[ht!]
    \centering
    
    \begin{subfigure}[b]{0.45\textwidth}
        \centering
        \begin{tikzpicture}[scale=0.9, xscale=.8]
            \draw[-latex,  semithick] (0,0)--node[pos=0, below left] {$0$} node[pos=1, right, semithick] {$\nu$} (8,0);
            \draw[-latex,  semithick] (0,0)--node[pos=1, left] {$m$} (0,5);

            \draw[semithick] (1,.1)--node[pos=1, below] {$\frac{1}{2}-\delta$} (1,-.1);
            \draw[blue,semithick] (2.5,.1)--node[pos=1, below] {$\beta_{BD}$} (2.5,-.1);
            \draw[semithick] (6,.1)--node[pos=1, below] {$\frac{1-\delta}{2}$} (6,-.1);
            \draw[semithick] (.1,4)--node[pos=1, left] {$\delta$} (-.1,4);

            \coordinate (o) at (0,0);
            \coordinate (a) at (0,4);
            \coordinate (b) at (1,0);
            \coordinate (c) at (2.5,0);
            \coordinate (d) at (6,0);

            \draw[dashed] (b)--node[pos=.5,circle,fill=black,inner sep=1pt] (p1) {} node[pos=1, circle, fill=black, inner sep=1pt] (p2) {} (d|-a);
            \draw[semithick, blue] (c)--(p1)--(p2)--(8,4);
        \end{tikzpicture}
        \caption{}
        \label{fig:sub-a}
    \end{subfigure}
    \hfill
    \begin{subfigure}[b]{0.45\textwidth}
        \centering
        \begin{tikzpicture}[scale=0.9, xscale=.8]
            \draw[-latex,  semithick] (0,0)--node[pos=0, below left] {$0$} node[pos=1, right, semithick] {$\nu$} (8,0);
            \draw[-latex,  semithick] (0,0)--node[pos=1, left] {$c$} (0,5);

            \draw[blue,semithick] (6,.1)--node[pos=1, below] {$\beta_{BD}$} (6,-.1);

            \coordinate (o) at (0,0);
            \coordinate (a) at (0,4);
            \coordinate (b) at (1,0);
            \coordinate (c) at (2.5,0);
            \coordinate (d) at (6,0);

            \draw[] (o)--node[pos=.7,circle,fill=black,inner sep=1pt] (p1) {} node[pos=1, circle, fill=black, inner sep=1pt] (p2) {} (d|-a);

            \coordinate (u) at (o-|p1);
            \draw[semithick] ($(u)+(0,.1)$)--node[pos=1, below] {$\frac{1}{2}-\delta$} ($(u)+(0,-.1)$);
            
            \draw[semithick, blue] (p2) to[out=-100,in=5] (o);
            \draw[dashed] (p1) to[out=-90,in=10, looseness=.7] (o);
        \end{tikzpicture}
        \caption{}
        \label{fig:sub-b}
    \end{subfigure}

    \caption{(a) Plot of $m(\nu, \delta)$ in Theorem 1.1 for a $\delta<\frac{1}{2}$. The dashed line is the previous bound of \cite[Theorem 1]{Dya19b}. (b) Plot of $c(\nu, \delta)$ in Theorem 1.2, again for a $\delta<\frac{1}{2}$. The solid line is the upper resolvent bound $c=2 \nu$ of \cite{DZ16}, \cite{BD17}. The dashed curve is the previous upper bound of \cite[Theorem 2]{Dya19b}.}
    \label{fig:twoside}
\end{figure}

Much of the theory of resonances for quantum open baker's maps parallels that for more complicated settings, including hyperbolic surfaces. In this context, we study resonances in annular regions and the appropriate analogue of (1.1) is the counting function

\[
\mathcal N_{k}(\nu):=\#\left\{\lambda\in \operatorname{Sp}(B_{N}):\, |\lambda|\ge M^{-\nu}\right\},\quad \nu\ge 0.
\]

In \cite[Theorem~3]{DJ17}, Dyatlov–Jin proved the estimate: For $\nu>0$ and  each $\varepsilon>0$,

\begin{equation*}
\mathcal N_{k}(\nu)\leq C\,N^{\min\left(2\left(\nu-\left(\tfrac{1}{2}-\delta\right)\right),\,\delta\right)+\varepsilon}, \quad k\to\infty. \tag{1.3} \label{eq:103}
\end{equation*}

\noindent Notice that, as with the bound of \cite[Theorem~1]{Dya19b}, the exponent $\min\left(2\left(\nu-\left(\tfrac{1}{2}-\delta\right)\right),\,\delta\right)$ matches $\tfrac{1}{2}-\delta$. However, using the fractal uncertainty principle for Cantor sets, it is also proved in \cite{DJ17} that certain improved spectral gaps exist beyond $\tfrac{1}{2}-\delta$. These include a gap $\beta>\tfrac{1}{2}-\delta$ and a gap $\beta_{E}=\tfrac{3}{4}\left(\tfrac{1}{2}-\delta\right)+\tfrac{\gamma_{\mathcal A}}{8}$, where $\gamma_{\mathcal A}>0$ is explicitly computable in terms of the additive energy of the system. See \cite[Section 3]{DJ17} for details.

We may now state the following fractal Weyl bound which improves over (1.3).

\begin{thm} \label{thm:1.3}
For each $\nu>0$ and $\varepsilon>0$, we have
\[
\mathcal N_{k}(\nu)\leq C\,N^{m(\nu,\delta)+\varepsilon}\quad \text{as } k\to\infty
\]
where
\[
m(\nu,\delta):=\min \left(4(\nu-\beta),\,4\bigl(\nu-\beta_{E}\bigr),\,2\Bigl(\nu-\bigl(\tfrac{1}{2}-\delta\bigr)\Bigr),\,\delta\right)
\]
for some $\beta>\tfrac{1}{2}-\delta$, and $\beta_{E}=\tfrac{3}{4}\bigl(\tfrac{1}{2}-\delta\bigr)+\tfrac{\gamma_{\mathcal A}}{8}$.
\end{thm}

We remark that only one of the exponents $4(\nu-\beta)$ or $4\left(\nu-\beta_{E}\right)$ is relevant for a given $\delta\in(0,1)$ with, in general, $\beta_{E}>\beta$ when $\delta\approx \tfrac{1}{2}$. As is the case for hyperbolic surfaces, since there is also a spectral gap with $\beta>0$ (see \cite[Theorem~1]{DJ17}), our estimate may or may not give new information when $\delta>\tfrac{1}{2}$, but always improves known estimates when $\delta\leq \tfrac{1}{2}$. Unlike the case of hyperbolic surfaces, however, since in this case we also match the additive energy gap, we can conclude with certainty that Theorem~\ref{thm:1.3} does give new information at least for some $\delta>\tfrac{1}{2}$. This is because the best known lower bound on the gap $\beta_{E}$ obtained by additive energy estimates is in general larger than the best known lower bound on the gap $\beta>\max \left(0,\tfrac{1}{2}-\delta\right)$ when $\delta\approx \tfrac{1}{2}$. See Figure 2, and see \cite[§3.4]{DJ17}, \cite[§5.2]{Dya19a} for more discussion.

\subsection{Further context and overview.} A corollary of our Theorem~\ref{thm:1.1} is an upper bound on the number of zeros of the Selberg zeta function since its zeros agree with resonances (see \cite{Bor16} for the definition and properties of this function). Using dynamical properties of the Selberg zeta function, \cite{GLZ04} shows that for any convex co-compact hyperbolic surface,

\begin{equation*}
\mathcal N(R,\nu)\leq C R^{\delta}\quad \text{as } R\to\infty, \tag{1.4} \label{eq:104}
\end{equation*}

\noindent for any $\nu$; the proof also works for convex co-compact Schottky manifolds in higher dimensions. Similar methods were subsequently used by Naud \cite{Nau14}, who improved (1.4) in the region $\nu<\frac{1-\delta}{2}$ by showing that there exists a function $\tau(\nu)$, positive in the range $\nu\in\left(\tfrac{1}{2}-\delta,\tfrac{1-\delta}{2}\right)$, for which

\begin{equation*}
\mathcal N(R,\nu)\leq C R^{\min(\delta-\tau(\nu),\,\delta)}\quad \text{as } R\to\infty. \tag{1.5} \label{eq:105}
\end{equation*}

\noindent Dyatlov \cite{Dya19b}, then improved (1.5) to
\begin{equation*}
\mathcal N(R,\nu)\leq C R^{\min\left(2\left(\nu-\left(\tfrac{1}{2}-\delta\right)\right),\,\delta\right)+\varepsilon}\quad \text{as } R\to\infty,
\end{equation*}
for any $\varepsilon>0$, which matches the Patterson–Sullivan spectral gap and makes Naud's bound explicit; he also gives an analogous bound for convex co-compact hyperbolic manifolds of higher dimensions. Dyatlov's approach is quite different from \cite{GLZ04,Nau14}, and involves applying the methods of Vasy, \cite{Vas13a,Vas13b} to reduce the problem to the estimate of a determinant function defined in terms of an approximate inverse to the semiclassically rescaled resolvent; see \cite[Section~1]{Dya19b} for an outline of his proof.

Our Theorem~\ref{thm:1.1} continues this string of improvements to (1.4), taking advantage of the improved spectral gap proved by Bourgain–Dyatlov \cite{BD17}, and combining their methods with those of \cite{Dya19b} to further improve the fractal Weyl bound. In particular, we follow much of the analysis of \cite{Dya19b}, but work with a different determinant function which allows us to reduce Theorem~\ref{thm:1.1} to an estimate of an oscillatory integral (see (2.16), (2.17)). This integral can be estimated via Fourier decay, and in Lemma 2.3 a slight modification of the argument of \cite{BD17} allows us to complete the proof of Theorem 1.1.

An analogous approach is used to prove Theorem~\ref{thm:1.3}; using a different determinant function than in \cite{DJ17}, we reduce the proof to an estimate of a quantity to which we may apply the method of proof of the fractal uncertainty principle of \cite[Theorem~1]{DJ17}. Here, the simpler structure present in the setting of quantum open baker's maps also allows an estimate in terms of the additive energy of the system.

\begin{figure}[ht!]
    \centering
    
    \begin{subfigure}[b]{0.45\textwidth}
        \centering
        \begin{tikzpicture}[scale=0.9, xscale=.8]
            \draw[-latex,  semithick] (0,0)--node[pos=0, below left] {$0$} node[pos=1, right, semithick] {$\nu$} (8,0);
            \draw[-latex,  semithick] (0,0)--node[pos=1, left] {$m$} (0,5);

            \draw[semithick] (1,.1)--node[pos=1, below] {$\frac{1}{2}-\delta$} (1,-.1);
            \draw[blue,semithick] (2.5,.1)--node[pos=1, below] {$\beta$} (2.5,-.1);
            \draw[semithick] (6,.1)--node[pos=1, below] {$\frac{1-\delta}{2}$} (6,-.1);
            \draw[semithick] (.1,4)--node[pos=1, left] {$\delta$} (-.1,4);

            \coordinate (o) at (0,0);
            \coordinate (a) at (0,4);
            \coordinate (b) at (1,0);
            \coordinate (c) at (2.5,0);
            \coordinate (d) at (6,0);

            \draw[dashed] (b)--node[pos=.5,circle,fill=black,inner sep=1pt] (p1) {} node[pos=1, circle, fill=black, inner sep=1pt] (p2) {} (d|-a);
            \draw[semithick, blue] (c)--(p1)--(p2)--(8,4);
        \end{tikzpicture}
        \caption{}
        \label{fig:sub-a}
    \end{subfigure}
    \hfill
    \begin{subfigure}[b]{0.45\textwidth}
        \centering
        \begin{tikzpicture}[scale=0.9, xscale=.8]
            \draw[-latex,  semithick] (0,0)--node[pos=0, below left] {$0$} node[pos=1, right, semithick] {$\nu$} (8,0);
            \draw[-latex,  semithick] (0,0)--node[pos=1, left] {$m$} (0,5);

            \draw[blue,semithick] (.5,.1)--node[pos=1, below] {$\beta_{E}$} (.5,-.1);
            \draw[semithick] (4,.1)--node[pos=1, below] {$\frac{1-\delta}{2}$} (4,-.1);
            \draw[semithick] (.1,4)--node[pos=1, left] {$\delta$} (-.1,4);

            \coordinate (o) at (0,0);
            \coordinate (a) at (0,4);
            \coordinate (b) at (.5,0);
            \coordinate (c) at (2.5,0);
            \coordinate (d) at (4,0);

            \draw[dashed] (0,.5)--node[pos=0,circle,fill=black,inner sep=1pt] (p0) {} node[pos=.5,circle,fill=black,inner sep=1pt] (p1) {} node[pos=1, circle, fill=black, inner sep=1pt] (p2) {} (d|-a);
            \draw[semithick, blue] (b)--(p1)--(p2)--(8,4);
        \end{tikzpicture}
        \caption{}
        \label{fig:sub-b}
    \end{subfigure}

    \caption{Plot of $m(\nu, \delta)$ from Theorem 1.3 when $(a)\; \delta<\frac{1}{2}$ and $\beta>\beta_E$, and (b) $\delta>\frac{1}{2},\left|\delta-\frac{1}{2}\right|$ small so that $\beta_E>\beta$. The dashed line in both is the previous bound of \cite[Theorem 3]{DJ17}.}
    \label{fig:twoside}
\end{figure}
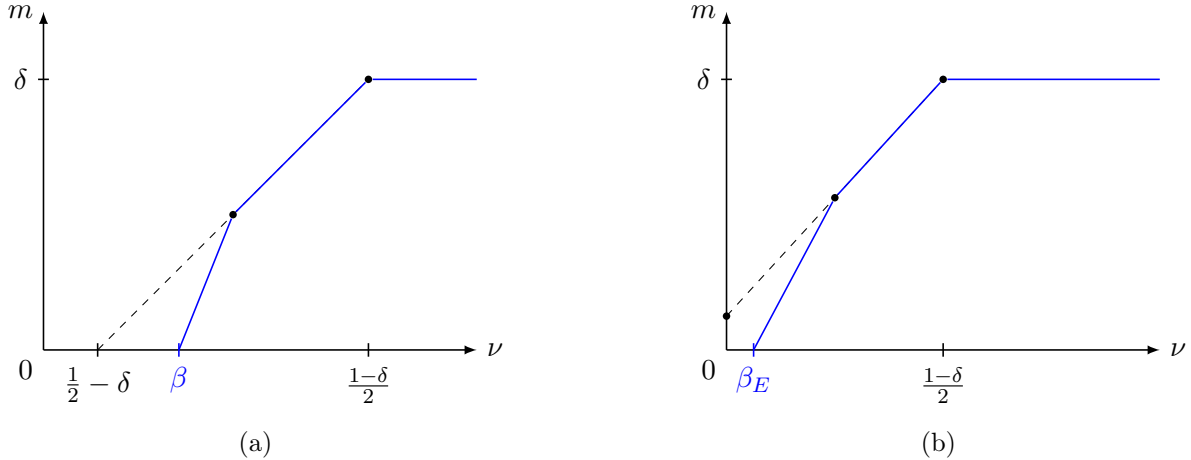

Finally, we mention some other fractal Weyl bounds. Using refined transfer operator methods of \cite{DZw18}, the paper \cite{Soa23} gives an improved fractal Weyl bound for hyperbolic surfaces for the global counting function
\begin{equation*}
\widetilde{\mathcal N}(R,\nu):=\#\left\{\lambda\ \text{resonance}:\ \operatorname{Re}\lambda\in[0,R],\ \operatorname{Im}\lambda\ge -\nu\right\}.
\end{equation*}
showing that
\begin{equation*}
\widetilde{\mathcal N}(R,\nu)\leq C R^{\,1+\delta-\frac{3\delta+2}{2\delta+2}(1-\delta-2\nu)+\varepsilon} \quad \text{as }  R\to\infty \tag{1.6} \label{eq:106}
\end{equation*}

\noindent With $m(\nu,\delta)$ as in Theorem~\ref{thm:1.1}, we can obtain a bound on $\widetilde{\mathcal N}(R,\nu)$ by integrating to get
\begin{equation*}
\widetilde{\mathcal N}(R,\nu)\leq C R^{\,1+m(\nu,\delta)+\varepsilon}\quad \text{as } R\to\infty.
\end{equation*}

\noindent This may or may not improve over (1.6) for some $\nu$, depending upon the size of the improvement $\beta_{BD}-\left(\tfrac{1}{2}-\delta\right)$. Vacossin, \cite{Vac23}, gives a bound similar to Naud's, (1.5), but for scattering by several strictly convex obstacles in the plane. His proof applies more generally to hyperbolic quantum monodromy operators, and the paper \cite{Tao24} shows that this includes certain asymptotically hyperbolic surfaces. Notably, using the fractal uncertainty principle, an improved spectral gap has also been given in these settings (see \cite{Vac24}, \cite{Tao24}) and it would be interesting if Vacossin's bound could be improved to match the improved spectral gap, analogously to our Theorem~\ref{thm:1.1}. We refer to \cite[§3.4]{Zwo17} for more information and history of fractal Weyl bounds.

The paper is organized as follows: The first half of the paper focuses on hyperbolic surfaces. In Section 2 we gather material from \cite{Dya19b}, \cite{BD17}, and establish some preliminary technical results. Section 3 gives the proofs of Theorems 1.1 and 1.2. Section 4 turns to quantum open baker's maps where we apply an analogous argument to that in Section 3 to give the proof of Theorem~\ref{thm:1.3}.

\vspace{0.2cm}

\noindent \textbf{Acknowledgements.} We thank Semyon Dyatlov for many helpful discussions, for suggesting the project, and for his encouragement throughout the writing of this paper. We also thank Tanya Christiansen for comments which improved the exposition. Finally, we thank Dan Cunningham for logistical help, and the Prison Mathematics Project - in particular Frank Connor, Ben Jeffers, and Tian An Wong - for handling the typing of the original manuscript.

\section{PRELIMINARIES}

The proofs of Theorems 1.1 and 1.2 combine the methods of \cite{Dya19b} and \cite{BD17}, and in this section we review definitions and gather some essential technical material. As mentioned in Section 1.1, the resolvent on a convex co-compact hyperbolic surface $M$ has a meromorphic continuation to $\mathbb{C}$ as an operator
\begin{equation*}
R(\lambda):=\left(-\Delta_{M}-\tfrac{1}{4}-\lambda^{2}\right)^{-1}: L_{c}^{2}(M)\to H_{\mathrm{loc}}^{2}(M),\qquad \lambda\in\mathbb{C},
\end{equation*}

\noindent see \cite{MM87}, \cite{Gu05}, and \cite{GuZw95} which also cover more general asymptotically hyperbolic manifolds. Here, we follow \cite[Section 2.2]{Dya19b} making use of an alternative method of proving meromorphic continuation of the resolvent due to Vasy, \cite{Vas13a}, \cite{Vas13b}. Our presentation will be minimalistic, gathering just what is needed to prove Theorems 1.1 and 1.2, referring to \cite{Dya19b} for more detail. 

We begin by defining the semiclassically rescaled resolvent
\begin{equation*}
R_{h}(\omega):=h^{-2} R(\lambda), \quad \omega:=h\lambda \in \Omega, 
\tag{2.1} \label{eq:201}
\end{equation*}

\noindent where we fix $\nu_0 > 0$ and set

$$
\Omega:=[1-2h,\,1+2h]+ih\,[ -\nu_{0},\,1] .
$$

\noindent We also define the extended, semiclassical differential operator as in \cite{Vas13a}, \cite{Vas13b}: 
\begin{equation*}
P_{h}(\omega)\in \Psi_{h}^2 \left(M_{\mathrm{ext}}\right), \quad
P_{h}(\omega)=\psi_{2}\left(-h^{2}\Delta_{M}-\frac{h^{2}}{4}-\omega^{2}\right)\psi_{1}\ \text{ on } M.
\tag{2.2} \label{eq:202}
\end{equation*}

\noindent Here $M_{\mathrm{ext}}$ is a compact surface without boundary, containing $M$ as an open subset, and $\psi_{1},\psi_{2}\in C^{\infty}(M)$ are certain nonvanishing functions. We refer to \cite{Vas13a}, \cite{Vas13b} for details on the construction of the operator $P_h(\omega)$, and to \cite{Zwo16}, \cite{DZw19} for simplified treatments of Vasy's method.

The advantage of working with $P_{h}(\omega)$ is the ease with which one can apply microlocal analytical methods; for our purposes, it is enough to simply quote a number of previously established properties of this operator:
\begin{enumerate}[label=(\roman*)]
    \item For any fixed $s>\tfrac{1}{2}+\nu_{0}$, $P_{h}(\omega): \mathcal{X}\to \mathcal{Y}$ is a holomorphic family of Fredholm operators, where
    \begin{equation*}
        \mathcal{X}:=\left\{u\in H_{h}^{s} \left(M_{\mathrm{ext}}\right):\ P_{h}(1)u\in H_{h}^{s-1} \left(M_{\mathrm{ext}}\right)\right\}, 
        \qquad 
        \mathcal{Y}:=H_{h}^{s-1} \left(M_{\mathrm{ext}}\right),
    \end{equation*}
    with
    \begin{equation*}
        \|u\|_{\mathcal{X}}^{2}:=\|u\|_{H_{h}^{s} \left(M_{\mathrm{ext}}\right)}^{2}+\|P_{h}(1)u\|_{H_{h}^{s-1} \left(M_{\mathrm{ext}}\right)}^{2}.
    \end{equation*}
    Note in particular that
    \begin{equation*}
        \|u\|_{\mathcal{X}}\leqslant C\,\|u\|_{H_{h}^{s+1} \left(M_{\mathrm{ext}}\right)}. \tag{2.3} \label{eq:203}
    \end{equation*}

 
    \item From (2.2), if $f\in C_{c}^{\infty}(M)$, then
    \begin{equation*}
        R_{h}(\omega)f=\left.\psi_{1}\Bigl(P_{h}(\omega)^{-1}\,\psi_{2}f\Bigr)\right|_{M}. \tag{2.4} \label{eq:204}
    \end{equation*}
    In particular, a bound on the number of poles of $P_{h}(\omega)^{-1}$ in a region $\Omega'\subset \Omega$ gives a bound on the number of poles of $R_{h}(\omega)$ in the same region.
\end{enumerate}

\begin{enumerate}[label=(\roman*),resume]
    \item We describe the following approximate inverse identity for $P_{h}(\omega)$ which is of crucial importance to both of the proofs of Theorems 1.1 and 1.2. It follows by combining \cite[Propositions~2.1, 3.1, and Lemma~3.2]{Dya19b}.
    
    \hspace{2em} Fix $\rho,\rho'\in(0,1)$ and $\varepsilon_{0}>0$ (in this paper, we only ever need $\rho$ near $1$, and in particular we will assume throughout that $\rho>\tfrac{1}{2}$). Then we can write
    \begin{equation*}
        I=Z(\omega)\,P_{h}(\omega)+A(\omega),\qquad \omega\in \Omega. \tag{2.5} \label{eq:205}
    \end{equation*}
    for $h$-dependent families of operators holomorphic in $\Omega$, satisfying
    \begin{equation*}
        Z(\omega):\mathcal{Y} \to \mathcal{X},\qquad 
        \|Z(\omega)\|_{\mathcal{Y}\to \mathcal{X}}\le C\,h^{-1-(\rho+\rho')\,(\nu_{0}+\varepsilon_{0})},\qquad \omega\in \Omega, \tag{2.6} \label{eq:206}
    \end{equation*}
    and
    \begin{equation*}
        A(\omega):\mathcal{X} \to \mathcal{X},\qquad 
        A(\omega)=J(\omega)\,A_{-}\,\tilde{A}\,A_{+}+\varepsilon(\omega),\qquad \omega\in \Omega. \tag{2.7} \label{eq:207}
    \end{equation*}
    where,
    \begin{itemize}
        \item For any $N>0$, $J(\omega),\,\varepsilon(\omega):H_{h}^{-N} \left(M_{\mathrm{ext}}\right)\to H_{h}^{N} \left(M_{\mathrm{ext}}\right)$ satisfy
        \begin{equation*}
            \|J(\omega)\|_{H_{h}^{-N}\to H_{h}^{N}}\le C_{N}\,h^{\rho'\bigl(h^{-1} \operatorname{Im}\omega-\varepsilon_{0}\bigr)} ,\qquad 
            \|\varepsilon(\omega)\|_{H_{h}^{-N}\to H_{h}^{N}}=O \left(h^{N}\right),\qquad \omega\in \Omega. \tag{2.8} \label{eq:208}
        \end{equation*}
        \item The operators $A_{-}:L^{2} \left(\mathbb{R}^{+}\times \mathbb{S}\right)\to L^{2} \left(M_{\mathrm{ext}}\right)$ and $A_{+}:L^{2} \left(M_{\mathrm{ext}}\right)\to
        L^{2} \left(\mathbb{R}^{+}\times \mathbb{S}\right)$ are bounded uniformly in $h$.
        
        \item If $(w,y)$ denote coordinates on $\mathbb{R}_{w}^{+}\times \mathbb{S}_{y}$, then
        \begin{equation*}
            \tilde{A}=\psi_{-}(y;h)\,\tilde{B}_{\psi}\,\psi_{+}(y;h)\,\psi_{0}(w;h)\,\tilde{\psi} \left(hD_{w}\right). \tag{2.9} \label{eq:209}
        \end{equation*}
        Here $\psi_{\pm}\in C^{\infty} \left(\mathbb{S};[0,1]\right)$ with
        \begin{equation*}
            \operatorname{supp}\psi_{+}\subset \Lambda_{\Gamma} (C_{1}h^{\rho}),\qquad 
            \operatorname{supp}\psi_{-}\subset \Lambda_{\Gamma} (C_{1}h^{\rho'}). \tag{2.10} \label{eq:210}
        \end{equation*}

        for some $C_{1}>0$ (where $\Lambda_{\Gamma}(\alpha):=\{\,y\in\mathbb{S}:\ |y-y'|\le \alpha \text{ for some } y'\in \Lambda_{\Gamma}\,\}$, and $|y-y'|$ denotes Euclidean distance in $\mathbb{R}^{2}$), 
$\psi_{0}\in C_{c}^{\infty} \left(\mathbb{R}^{+};[0,1]\right)$ with
\[
\operatorname{supp}\psi_{0}\subset \bigl[\,1-C_{1}h^{\rho},\ 1+C_{1}h^{\rho}\,\bigr],
\]
$\widetilde{\psi}\in C_{c}^{\infty}(\mathbb{R};[0,1])$ is independent of $h$, and $\widetilde{B}_{\psi}$ is the operator on $L^{2} \left(\mathbb{R}^{+}\times \mathbb{S}\right)$ given by
\[
\bigl(\widetilde{B}_{\psi}f\bigr)(w,y)
=(2\pi h)^{-1/2}\int_{\mathbb{S}}\left|\frac{y-y'}{2}\right|^{2iw/h}\,\psi(y,y')\,f(w,y')\,dy',
\qquad f\in L^{2} \left(\mathbb{R}^{+}\times \mathbb{S}\right),
\]
with $\psi\in C^{\infty} \left(\mathbb{S}\times \mathbb{S}\right)$ an $h$-independent function satisfying 
$\operatorname{supp}\psi\ \cap\ \{\,y=y'\,\}=\varnothing$.

Finally, we mention that $A(\omega)$ is Hilbert–Schmidt on $\mathcal{X}$ \cite[Proposition~3.1]{Dya19b}.
    \end{itemize}
\end{enumerate}

In the next section, we will use $A(\omega)$ to define a determinant function which we will use to prove Theorem~\ref{thm:1.1}. We then make estimates on this determinant function using methods from both \cite{Dya19b} and \cite{BD17}, and we describe some relevant quantities used in the latter paper now.

First, we note that to this point we have followed \cite{Dya19b} working in the unit disk model of the hyperbolic space (for which the limit set is a subset of the circle $\mathbb{S} \subset \mathbb{R}^{2}$); on the other hand, \cite{BD17} uses the half-plane model (for which the limit set is a compact subset of $\mathbb{R}$). The two models are equivalent and we could easily translate the material of \cite{BD17} to the disk model. However, since we isolate our use of the methods of \cite{BD17} to the single technical lemma below, it is more convenient to describe their results in the original context (i.e. the half-plane model), and to explain only the relationship between the two models relevant to our needs (see \cite{Bor16} for more information).

Assume $\delta \in(0,1)$ (which is the case in all interesting situations). Then by rotational symmetry, we may assume without loss of generality that $(1,0)\notin \Lambda_{\Gamma}$ (since $\Lambda_{\Gamma}$ is a strict subset of $\mathbb{S}$). We then use the standard transformation
$$
\sigma(x_{1},x_{2})
:=\left(
\frac{x_{1}^{2}+x_{2}^{2}-1}{\,x_{1}^{2}+(x_{2}+1)^{2}\,}\,,\ 
\frac{-2x_{1}}{\,x_{1}^{2}+(x_{2}+1)^{2}\,}
\right),
$$
which maps $\{(x_{1},x_{2})\in\mathbb{R}^{2}:x_{2}>0\}$ with (hyperbolic) metric 
$
ds^{2}=\frac{dx_{1}^{2}+dx_{2}^{2}}{x_{2}^{2}}
$
isometrically onto $\{(x_{1},x_{2})\in\mathbb{R}^{2}:x_{1}^{2}+x_{2}^{2}<1\}$ with (hyperbolic) metric 
$
ds^{2}=4\,\frac{dx_{1}^{2}+dx_{2}^{2}}{\bigl(1-(x_{1}^{2}+x_{2}^{2})\bigr)^{2}},
$
and relates $\Lambda_{\Gamma}\subset \mathbb{S}$ to
\[
\Lambda_{\Gamma}^{\mathbb{R}}:=\sigma^{-1}(\Lambda_{\Gamma})
\]

\noindent which is a compact subset of $\mathbb{R}$. As the notation suggests, \(\Lambda_{\Gamma}^{\mathbb{R}}\) is exactly the limit set for the surface \(M=\Gamma\backslash\mathbb{H}\) in the model obtained from the isometry \(\sigma\).

We also note a couple of facts for future use. Firstly, for any \(C_{1}>0\), there exists \(C_{2}>0\) such that
\begin{equation*}
\Lambda_{\Gamma} \left(C_{1}\alpha\right)\ \subset\ \sigma \left(\Lambda_{\Gamma}^{\mathbb{R}} \left(C_{2}\alpha\right)\right). \tag{2.11} \label{eq:211}
\end{equation*}

\noindent for all \(\alpha>0\) small (here \(\Lambda_{\Gamma}^{\mathbb{R}}(\alpha):=\{\,x\in\mathbb{R}:\ |x-x'|\le \alpha \text{ for some } x'\in \Lambda_{\Gamma}^{\mathbb{R}}\,\}\)). Secondly, there exists \(C>0\) such that for small \(\alpha,\alpha'>0\) and \(x_{0}\in \Lambda_{\Gamma}^{\mathbb{R}}\),
\begin{equation*}
C^{-1}\,\alpha^{1-\delta}\,(\alpha')^{\delta}
\ \le\
\bigl|\Lambda_{\Gamma}^{\mathbb{R}}(\alpha)\cap\{\,|x-x_{0}|\le \alpha'\,\}\bigr|
\ \le\
C\,\alpha^{1-\delta}\,(\alpha')^{\delta}.
\tag{2.12}\label{eq:212}
\end{equation*}

\noindent This is a version of Ahlfors–David regularity for the set \(\Lambda_{\Gamma}^{\mathbb{R}}\), and is proved in exactly the same way as \cite[(5.3)]{DZ16}; see \cite[§7.2]{DZ16}.

 A fundamental tool used in the study of convex co-compact hyperbolic surfaces is the Patterson–Sullivan measure, which is a probability measure on \(\Lambda_{\Gamma}^{\mathbb{R}}\) that we denote by \(\mu\). We recall the following crucial property of this measure; we refer to \cite[§14.1]{Bor16} and \cite[§2.4]{BD17} for more information on \(\mu\).

\medskip

\begin{thm} \label{thm:2.1}
    (\cite[Theorem~2]{BD17}). 
Let \(M\) be a convex co-compact hyperbolic surface, and let \(\delta\in(0,1)\) be the dimension of the limit set \(\Lambda_{\Gamma}^{\mathbb{R}}\).
Assume that
\[
\varphi\in C^{2}(\mathbb{R};\mathbb{R}),\qquad g\in C^{1}(\mathbb{R};\mathbb{C})
\]
are functions satisfying the following bounds for some constant \(C_{\varphi,g}\):
\[
\|\varphi\|_{C^{2}}+\|g\|_{C^{1}}\le C_{\varphi,g},\qquad 
\inf_{\,\Lambda_{\Gamma}^{\mathbb{R}}}\bigl|\varphi'(x)\bigr|\ \ge\ C_{\varphi,g}^{-1}.
\]
Then there exists \(\beta_{F}>0\) depending only on \(\delta\), and \(C>0\) depending only on \(M\) and \(C_{\varphi,g}\), such that
\begin{equation*}
\Bigl|\int_{\Lambda_{\Gamma}^{\mathbb{R}}} e^{i\xi\,\varphi(x)}\,g(x)\,d\mu(x)\Bigr|
\ \le\ C\,|\xi|^{-\beta_{F}/2}
\qquad \text{for all }|\xi|>1. \tag{2.13}\label{eq:213}
\end{equation*}

\end{thm}

\noindent (Compared with the statement in \cite{BD17}, we followed \cite[§5.1]{Dya19a} writing $\varepsilon_{1}=\beta_{F}/2$.) This theorem is then used in \cite{BD17} to give a spectral gap of size
\[
\beta_{\mathrm{BD}}=\tfrac{1}{2}-\delta+\tfrac{\beta_{F}}{8}.
\]

\noindent It is this quantity which appears in our Theorems 1.1 and 1.2.

As remarked immediately after the statement of this theorem in \cite{BD17}, \eqref{eq:213} implies that the Fourier dimension of $\Lambda_{\Gamma}^{\mathbb{R}}$ is bounded below by $\beta_{F}$. Since the Hausdorff dimension of a set is always greater than or equal to the Fourier dimension, we obtain that $\beta_F$ must satisfy $\beta_{F}\le \delta$.

The proof of Theorem~\ref{thm:2.1} uses the following technical lemma, which is also an important tool for us in the proof of Theorems 1.1 and 1.2.

\begin{lm}\label{lm:2.2}(\cite[Lemma~4.2]{BD17}). For \(0<h<1\), define the function \(F_{h}(x)\) as the convolution of the Patterson–Sullivan measure \(\mu\) with the rescaled uniform measure on \([-2h,2h]\):
\begin{equation*}
F_{h}(x):=\frac{1}{4\,h^{\delta}}\;\mu\bigl([x-2h,\,x+2h]\bigr).
\tag{2.14} \label{eq:214}
\end{equation*}

\noindent Then for some constant $C_{\Gamma}>0$ depending only on $\Gamma$,
\begin{equation*}
F_{h}\ \ge\ C_{\Gamma}^{-1}\quad \text{on }\Lambda_{\Gamma}^{\mathbb{R}}(h). \tag{2.15} \label{eq:215}
\end{equation*}
\end{lm}

We now prove an important estimate for later use. We give the proof here since it is the only part of the paper which uses some of the above material in an essential way.

Fix constants $c_{1},c_{2}>0$, let $G\in C_{c}^{\infty}(\mathbb{R}^{2})$ with $\operatorname{supp}G\cap\{x=y\}=\varnothing$, and set
\begin{equation*}
I(h):=\iint_{\Lambda_{\Gamma}^{\mathbb{R}}(c_1h^{\rho})^2} \bigl|K(x,w;h)\bigr|^{2}\,dx\,dw. \tag{2.16} \label{eq:216}
\end{equation*}
where
\begin{equation*}
K(x,w;h)
:=\int_{\mathbb{R}} F_{c_{2}h^{\rho^{'}}}(\eta)\,
\exp \left(\frac{i}{h}\bigl(\Phi(x,\eta)-\Phi(w,\eta)\bigr)\right)\,
G(x,\eta)\,G(w,\eta)\,d\eta, \tag{2.17} \label{eq:217}
\end{equation*}
and
\[
\Phi(x,\eta)=2\log|x-\eta|.
\]

In Section 3, we reduce the main part of the proofs of Theorems 1.1 and 1.2 to the following estimate, which bounds quantities of the form $I(h)$ in terms of the Fourier decay exponent of Bourgain–Dyatlov in Theorem~\ref{thm:2.1}. The proof uses a combination of ideas from \cite[Propositions 4.1 and 4.3]{BD17}.

\begin{lm}\label{lm:2.3}
With $I(h)$ as defined in \eqref{eq:216}, \eqref{eq:217}, there is $C>0$ such that
\begin{equation*}
I(h) \le C\, h^{2\rho'(1-\delta)+2\rho(1-\delta)+\beta_{F}/2}
\end{equation*}
for small $h$.
\end{lm}

\begin{proof}
Recalling the definition of $F_h$ in \eqref{eq:214}, we may write for an appropriate constant $C$,
\begin{align*}
K(x,w;h)
&= C\,h^{-\rho'\delta}\int_{\mathbb{R}}
  \mu \left(\bigl[\eta-2c_2 h^{\rho'},\,\eta+2c_2 h^{\rho'}\bigr]\right)
  \exp \left(\frac{i}{h}\bigl(\Phi(x,\eta)-\Phi(w,\eta)\bigr)\right)
  G(x,\eta)\,G(w,\eta)\,d\eta \\
&= C\,h^{-\rho'\delta}\int_{-2c_2 h^{\rho'}}^{2c_2 h^{\rho'}} 
   \int_{\Lambda_{\Gamma}^{\mathbb{R}}}
   \exp \left(\frac{i}{h}\bigl(\Phi(x,\eta-s)-\Phi(w,\eta-s)\bigr)\right)
   G(x,\eta-s)\,G(w,\eta-s)\,d\mu(\eta)\,ds .
\end{align*}

\noindent Thus, using Schwarz inequality, we have
\begin{equation*}
I(h)\le C\,h^{\rho'-2\rho'\delta}\int_{-2c_2 h^{\rho'}}^{2c_2 h^{\rho'}} 
\left(
  \iint_{\Lambda_{\Gamma}^{\mathbb{R}}(c_1 h^{\rho})^{2}}
  \bigl|\tilde{K}_s(x,w;h)\bigr|^{2}\,dx\,dw
\right) ds,
\tag{2.18}\label{eq:218}
\end{equation*}
where
\begin{equation*}
\tilde{K}_s(x,w;h)
= \int_{\Lambda_{\Gamma}^{\mathbb{R}}}
   \exp \left(\frac{i}{h}\bigl(\Phi(x,\eta-s)-\Phi(w,\eta-s)\bigr)\right)
   G(x,\eta - s)\,G(w,\eta - s)\,d\mu(\eta).
\tag{2.19}\label{eq:219}
\end{equation*}

We now mimic the end of the proof of \cite[Proposition 4.1]{BD17}, 
and, as in that proof, we may use a partition of unity for $G$ to reduce to the case in which 
$\operatorname{supp} G \subset J_{1}\times J_{2} \subset I_{1}\times I_{2} \Subset I_{1}^{\prime}\times I_{2}^{\prime}$
for $J_{1},J_{2}$ closed intervals, and $I_{1}, I_{2}, I_{1}^{\prime}, I_{2}^{\prime}$ open ones, with 
$I_{1}^{\prime}\cap I_{2}^{\prime}=\varnothing$. 
For $x,w \in \Lambda_{\Gamma}^{\mathbb{R}}(c_1 h^{\rho})\cap J_{1}$ and $s\in[-2c_{2}h^{\rho'},\,2c_{2}h^{\rho'}]$, define functions 
$\varphi_{x,w,s},\, g_{x,w,s}$ on $I_{2}$ by
\begin{equation*}
\begin{aligned}
\Phi(x,\eta - s) - \Phi(w,\eta - s) &= (x-w)\,\varphi_{x,w,s}(\eta),\\
g_{x,w,s}(\eta) &= G(x,\eta - s)\,G(w,\eta - s), \qquad \eta \in I_{2}.
\end{aligned}
\tag{2.20}\label{eq:220}
\end{equation*}

\noindent Then, using the form of $\Phi$ and the fact that $G$ is $C^{1}$, we easily obtain
\begin{equation*}
\|\varphi_{x,w,s}\|_{C^{2}(I_{2})}+\|g_{x,w,s}\|_{C^{1}(I_{2})}\le C,
\qquad 
\inf_{I_{2}}\bigl|\partial_{\eta}\varphi_{x,w,s}\bigr|\ge C^{-1}
\end{equation*}
for a constant $C$ independent of $x,w,s$. We can extend $\varphi_{x,w,s}, g_{x,w,s}$ to functions compactly supported on $\mathbb{R}$, and such that

$$
\|\varphi_{x,w,s}\|_{C^{2}(\mathbb{R})}+\|g_{x,w,s}\|_{C^{1}(\mathbb{R})}\le C,
\qquad 
\inf_{\Lambda_{\Gamma}^{\mathbb{R}}}\bigl|\partial_{\eta}\varphi_{x,w,s}\bigr|\ge C^{-1}
$$
for a possibly different $C$, but still independent of $x,w,s$. 
But \eqref{eq:219} and \eqref{eq:220} show that
$$
\tilde{K}_{s}(x,w;h)
=\int_{\Lambda_{\Gamma}^{\mathbb{R}}}
   \exp \bigl(i\,\xi\,\varphi_{x,w,s}(\eta)\bigr)\,
   g_{x,w,s}(\eta)\,d\mu(\eta),
\qquad 
\xi:=\frac{x-w}{h},
$$

\noindent so we may apply Theorem~\ref{thm:2.1} to get
\begin{equation*}
\bigl|\tilde{K}_{s}(x,w;h)\bigr|^{2}\le C\left|\frac{x-w}{h}\right|^{-\beta_{F}},
\quad x,w\in \Lambda_{\Gamma}^{\mathbb{R}}(c_1h^{\rho})\cap J_{1},\quad |x-w|>h,
\tag{2.21}\label{eq:221}
\end{equation*}
uniformly in $s\in[-2c_{2}h^{\rho'},\,2c_{2}h^{\rho'}].$

Now, for any $w\in \Lambda_{\Gamma}^{\mathbb{R}}(c_{1}h^{\rho})\cap J_{1}$ and $s\in[-2c_{2}h^{\rho'},\,2c_{2}h^{\rho'}]$, we have

\begin{align*}
\int_{\Lambda_{\Gamma}^{\mathbb{R}}(c_{1}h^{\rho})}\bigl|\tilde{K}_{s}(x,w;h)\bigr|^{2}\,dx
&= \int_{\Lambda_{\Gamma}^{\mathbb{R}}(c_{1}h^{\rho})\cap\{|x-w|\le h^{1/2}\}}
   \bigl|\tilde{K}_{s}(x,w;h)\bigr|^{2}\,dx \\
&\quad + \int_{\Lambda_{\Gamma}^{\mathbb{R}}(c_{1}h^{\rho})\cap\{|x-w|> h^{1/2}\}}
   \bigl|\tilde{K}_{s}(x,w;h)\bigr|^{2}\,dx.
\tag{2.22}\label{eq:222}
\end{align*}

\noindent To bound the first integral, we note that for any $w\in \Lambda_{\Gamma}^{\mathbb{R}}(c_{1}h^{\rho})$ we can find 
$w_{0}\in \Lambda_{\Gamma}^{\mathbb{R}}$ so that 
$\{x:|x-w|\le h^{1/2}\}\subset \{x:|x-w_{0}|<2h^{1/2}\}$ for small $h$
(here we use $\rho>\tfrac{1}{2}$).
By \eqref{eq:212}, the Lebesgue measure of 
$\Lambda_{\Gamma}^{\mathbb{R}}(c_1h^{\rho})\cap\{|x-w_{0}|<2h^{1/2}\}$ 
is bounded by $C h^{\rho(1-\delta)+\delta/2}$.
For the second integral, we use \eqref{eq:221} to bound it by 
$C h^{\rho(1-\delta)+\beta_{F}/2}$.
Then from \eqref{eq:222} and \eqref{eq:218}, and recalling that $\beta_{F}\le \delta$, we have
\begin{equation*}
I(h)\le C\,h^{2\rho'(1-\delta)+2\rho(1-\delta)+\beta_{F}/2} .
\end{equation*}
as claimed. 

\end{proof}

\section{PROOFS OF THEOREMS 1.1 AND 1.2}

\subsection{Improved fractal Weyl bound matching the improved spectral gap.} In this subsection, we will give the proof of Theorem~\ref{thm:1.1}. Along the way, we will also prove some bounds that will be used in the proof of Theorem~\ref{thm:1.2} in the next subsection.

We begin by defining the determinant function
\begin{equation*}
F(\omega):=\operatorname{det} \left(I-A(\omega)^{4}\right), \quad \omega \in \Omega.
\tag{3.1}\label{eq:301}
\end{equation*}
Here $\rho,\rho'\in(0,1)$, $\varepsilon_{0}>0$, $\nu_{0}>0$ are fixed but arbitrary, $A(\omega)$ is as defined in Section~2, and we recall that
\begin{equation*}
\Omega:=[1-2h,\,1+2h]+ih\,[-\nu_{0},\,1] .
\tag{3.2}\label{eq:302}
\end{equation*}
This function is well-defined since $A(\omega)$ is Hilbert–Schmidt (see Section~2), and hence $A(\omega)^{4}$ is trace-class. Moreover, it is not hard to see from Section~2 (iii) that
\begin{equation*}
\|A(\omega)\|_{\mathcal{X}\to\mathcal{X}} \le C\, h^{\rho'\,\bigl(h^{-1}\operatorname{Im}\omega-\varepsilon_{0}\bigr)}, \qquad \omega\in\Omega .
\tag{3.3}\label{eq:303}
\end{equation*}
and thus $\bigl(I-A(\omega_{0})^{4}\bigr)^{-1}$ certainly exists for some $\omega_{0}\in\Omega$, provided $\varepsilon_{0}$ was chosen sufficiently small. By analytic Fredholm theory (see, e.g., \cite[Theorem C.8]{DZw19}), $\bigl(I-A(\omega)^{4}\bigr)^{-1}$ exists as a meromorphic family in $\Omega$, with poles of finite rank; thus $F$ is holomorphic in $\Omega$ and its zeros are precisely these poles, with agreement of multiplicities (see \cite[Section B.5]{DZw19}). From \eqref{eq:205}, we see that
\begin{equation*}
P_{h}(\omega)^{-1}=\bigl(I-A(\omega)^{4}\bigr)^{-1}\bigl(I+A(\omega)+A(\omega)^{2}+A(\omega)^{3}\bigr)\,Z(\omega), \qquad \omega\in\Omega,
\end{equation*}
so the poles of $\bigl(I-A(\omega)^{4}\bigr)^{-1}$ agree with the poles of $P_{h}(\omega)^{-1}$. Since \eqref{eq:204} shows that any pole of $R_{h}(\omega)$ is a pole of $P_{h}(\omega)^{-1}$, we conclude from \eqref{eq:201} that for any $\nu<\nu_{0}$, if we put $h:=R^{-1}$, then an upper bound on the number of zeros of $F$ in the region
\begin{equation*}
\Omega' := [1,\,1+h]+ih\,\Bigl[-\nu,\,\tfrac{1}{2}\Bigr] \subset \Omega
\tag{3.4}\label{eq:304}
\end{equation*}
implies an upper bound on
\[
\mathcal{N}(R,v)=\#\{\lambda \text{ resonance} : \operatorname{Re}\lambda \in [R,R+1],\ \operatorname{Im}\lambda \ge -\nu\}.
\]

To prove Theorem~\ref{thm:1.1}, it suffices to prove the estimate
\[
\mathcal{N}(R,v) \le C\,R^{4\,(\nu-\beta_{BD})+\varepsilon} \qquad \text{as } R\to\infty,
\]
since combining this with \cite[Theorem 1]{Dya19b} proves the theorem, and we do this by estimating the number of zeros of $F$ in $\Omega'$ using a variant of Jensen's inequality. To carry this out, we will need an upper bound on $F(\omega)$ in the larger region $\Omega$, and we begin
by recalling the standard estimate \cite[(B.5.19)]{DZw19}:
\begin{equation*}
\log|F(\omega)| \le \bigl\|A(\omega)^{4}\bigr\|_{\operatorname{tr}(\mathcal{X})}, \qquad \omega\in\Omega .
\tag{3.5}\label{eq:305}
\end{equation*}

We estimate $\bigl\|A(\omega)^{4}\bigr\|_{\operatorname{tr}(\mathcal{X})}$ in two steps, the first of which is the following.

\begin{prop}\label{prop:3.1}
With $\rho,\rho'\in(0,1)$, $\varepsilon_{0}>0$, and $\nu_{0}>0$ fixed, the operator $A(\omega)$ satisfies
\begin{equation*}
\bigl\|A(\omega)^{4}\bigr\|_{\operatorname{tr}(\mathcal{X})}
\le C\,\|J(\omega)\|_{L^{2}\rightarrow L^{2}}^{4}\,
\operatorname{tr} \bigl((\tilde{A}^{*}\tilde{A})^{2}\bigr)
+ O \left(h^{\infty}\right), \qquad \omega\in\Omega,
\tag{3.6}\label{eq:306}
\end{equation*}
where $J(\omega)$ and $\tilde{A}$ are as in \eqref{eq:208} and \eqref{eq:209}, respectively.
\end{prop}

\begin{proof}
We use \cite[Theorem 1.13 and Corollary 1.10]{Sim79} to write
\begin{equation*}
\bigl\|A(\omega)^{4}\bigr\|_{\operatorname{tr}(\mathcal{X})}
\le \sum_{j=0}^{\infty} s_{j}\bigl(A(\omega)\bigr)^{4},
\end{equation*}
where $s_{j}(A(\omega))$ are the singular values of $A(\omega)$ (note that $\sum_{j=0}^{\infty}s_{j}(B)^{4}<\infty$ whenever $B$ is Hilbert–Schmidt, since the ideal property of the Hilbert–Schmidt class gives
$\sum_{j=0}^{\infty}s_{j}(B)^{4}=\|B^{*}B\|_{HS}^{2}\le \|B\|^{2}\|B\|_{HS}^{2}$).
By the second part of \eqref{eq:208}, $\varepsilon(\omega)$ is Hilbert–Schmidt (hence so is $J(\omega)A_{-}\tilde{A}A_{+}$), and using \eqref{eq:203} it satisfies
\begin{equation*}
\|\varepsilon(\omega)\|_{HS(\mathcal{X})}
\le C\,\|\varepsilon(\omega)\|_{HS\bigl(H_{h}^{s}(M_{\mathrm{ext}})\to H_{h}^{s+1}(M_{\mathrm{ext}})\bigr)}
= O(h^{\infty}), \qquad \omega\in\Omega .
\end{equation*}
Thus, recalling \eqref{eq:207} and using \cite[Theorem~1.21]{Sim79}, we obtain
\begin{align*}
\sum_{j=0}^{\infty}s_{j}\bigl(A(\omega)\bigr)^{4}
&\le C \sum_{j=0}^{\infty}s_{j}\bigl(J(\omega)A_{-}\tilde{A}A_{+}\bigr)^{4} + O(h^{\infty}) \\
&\le C\,\|J(\omega)\|_{L^{2}\to L^{2}}^{4}\sum_{j=0}^{\infty}s_{j}(\tilde{A})^{4} + O(h^{\infty}),
\tag{3.7}\label{eq:307}
\end{align*}

\noindent where in the latter step we used that $s_{j}(A B C) \leq \|A\|\,\|C\|\, s_{j}(B)$ when $B$ is compact and $A$ and $C$ are bounded (see \cite[Proposition B.15]{DZw19}). Since $\sum_{j=0}^{\infty} s_{j}(\tilde{A})^{4}=\operatorname{tr} \left((\tilde{A}^{*}\tilde{A})^{2}\right)$, we conclude the proof.

\end{proof}

Thus, to bound $F$ from above, we need to bound $\operatorname{tr} \left((\tilde{A}^{*}\tilde{A})^{2}\right)$, and we will do this in part by reducing to the quantity estimated in Lemma \ref{lm:2.3}. This is done in the following:

\begin{prop}\label{prop:3.2}
There exists $C>0$ such that
\begin{equation*}
\operatorname{tr} \left((\tilde{A}^{*}\tilde{A})^{2}\right)
\leq C\, h^{\,2 \rho-4+2 \rho(1-\delta)+2 \rho^{\prime}(1-\delta)+\beta_{F} / 2}
\end{equation*}
for small $h$.
\end{prop}

\begin{proof}
    
 Let $C_{2}>0$ be such that \eqref{eq:211} holds with $\alpha=h^{\rho'}$, where $C_{1}$ is as in \eqref{eq:210}. Next, we define
\[
\tilde{F}(y;h):=F_{C_{2} h^{\rho'}} \left(\sigma^{-1}(y)\right)
\]
with $F_{h}$ defined in \eqref{eq:214}. Using \eqref{eq:215}, we may write
\begin{equation*}\tag{3.8}\label{eq:308}
s_{j}(\tilde{A}) \leq C\, s_{j}(\hat{A})
\end{equation*}
where
\begin{equation*}\tag{3.9}\label{eq:309}
\hat{A}:=\sqrt{\tilde{F}(y;h)}\, \tilde{B}_{\psi}\, \psi_{+}(y;h)\, \psi_{0}(w;h)\, \tilde{\psi} \left(h D_{w}\right)
\end{equation*}
(this is just $\tilde{A}$ with $\psi_{-}$ replaced by $\sqrt{\tilde{F}}$; to prove \eqref{eq:308} we used that for $A$ bounded and $B$ compact, $s_{j}(A B) \leq \|A\|\, s_{j}(B)$, and the estimate $\left|\psi_{-}\right|=\left|\psi_{-}\right|(C_{\Gamma}\tilde{F})^{-1/2}(C_{\Gamma}\tilde{F})^{1/2} \le C\tilde{F}^{1/2}$ by Lemma \ref{lm:2.2}). Thus,
\begin{equation*}\tag{3.10}\label{eq:310}
\operatorname{tr} \left((\tilde{A}^{*}\tilde{A})^{2}\right) \leq C\, \operatorname{tr} \left((\hat{A}^{*}\hat{A})^{2}\right),
\end{equation*}

\noindent so we must estimate

\begin{equation*}\tag{3.11}\label{eq:311}
\operatorname{tr} \left((\hat{A}^{*}\hat{A})^{2}\right)
= \int_{\mathbb{R}^{+}} \int_{\mathbb{S}} \int_{\mathbb{R}^{+}} \int_{\mathbb{S}}
\left|\operatorname{ker}(\hat{A}^{*}\hat{A})(w,y,w',y')\right|^{2}\,dw\,dy\,dw'\,dy'.
\end{equation*}

By \eqref{eq:309}, we calculate
\begin{equation*}
\begin{aligned}
\operatorname{ker}(\hat{A}^{*}\hat{A})(w,y,w',y')
&= (2\pi h)^{-3}\,\psi_{+}(y;h)\,\psi_{+}(y';h)
\int_{\mathbb{R}^{+}} \mathcal{F}(\tilde{\psi}) \left(\frac{w'-\alpha}{h}\right)
\,\mathcal{F}(\tilde{\psi})\left(\frac{\alpha-w}{h}\right)\,\psi_{0}^{2}(\alpha;h) \\
&\qquad\cdot\Bigg(\int_{\mathbb{S}} \tilde{F}(z;h)\,
\left|\frac{z-y'}{z-y}\right|^{2 i \alpha / h}\,
\psi(z,y)\,\psi(z,y')\,dz\Bigg)\,d\alpha,
\end{aligned}
\end{equation*}
where \(\mathcal{F}(f)(\xi):=\int e^{-i\xi x} f(x)\,dx\).
Since \(\operatorname{supp}\psi_{0}\subset \bigl[\,1-C_{1}h^{\rho},\,1+C_{1}h^{\rho}\,\bigr]\), we use Schwarz inequality to estimate
\begin{equation*}\tag{3.12}\label{eq:312}
\begin{aligned}
\bigl|\operatorname{ker}(\hat{A}^{*}\hat{A})(w,y,w',y')\bigr|^{2}
&\le C\,h^{\rho-6}\,\psi_{+}(y;h)\,\psi_{+}(y';h)
\int_{\mathbb{R}}
\Bigl|\mathcal{F}(\tilde{\psi}) \left(\frac{w'-\alpha}{h}\right)\Bigr|^{2}
\Bigl|\mathcal{F}(\tilde{\psi}) \left(\frac{\alpha-w}{h}\right)\Bigr|^{2} \\
&\qquad\cdot \psi_{0}^{4}(\alpha;h)\,
\bigl|K(\alpha,y,y';h)\bigr|^{2}\,d\alpha,
\end{aligned}
\end{equation*}
where we defined
\[
K(\alpha,y,y';h):=\int_{\mathbb{S}} \tilde{F}(z;h)\,
\left|\frac{z-y'}{z-y}\right|^{2 i \alpha / h}\,
\psi(z,y)\,\psi(z,y')\,dz.
\]

\noindent Integrating \eqref{eq:312} and using that $\int_{\mathbb{R}}\bigl|\mathcal{F}(\tilde{\psi})\bigl((\alpha-w)/h\bigr)\bigr|^{2}\,dw \le h\,\|\tilde{\psi}\|_{L^{2}}^{2}$ and $\int_{\mathbb{R}}\bigl|\mathcal{F}(\tilde{\psi})\bigl((w'-\alpha)/h\bigr)\bigr|^{2}\,dw' \le h\,\|\tilde{\psi}\|_{L^{2}}^{2}$, we obtain from \eqref{eq:311} that (recalling the support condition on $\psi_{+}$ stated above),

\begin{equation*}\tag{3.13}\label{eq:313}
\operatorname{tr}\left((\hat{A}^{*}\hat{A})^2 \right) \; \le \; Ch^{\rho-4} \int_{\mathbb{R}^+} \psi_0^4(\alpha; h) \left( \iint_{\Lambda_\Gamma(C_1h^{\rho})^2} \abs{K(\alpha, y, y'; h)}^2 dydy' \right)d\alpha.
\end{equation*}

Using (\ref{eq:211}), we have 
\begin{equation*}\tag{3.14}\label{eq:314}
\iint_{\Lambda_{\Gamma}(C_{1}h^{\rho})^{2}}\bigl|K(\alpha,y,y';h)\bigr|^{2}\,dy\,dy'
\;\le\;
\iint_{\sigma\bigl(\Lambda_{\Gamma}^{\mathbb{R}}(C_{2}h^{\rho})\bigr)^{2}}
\bigl|K(\alpha,y,y';h)\bigr|^{2}\,dy\,dy'.
\end{equation*}

\noindent Changing variables using $\sigma$, noting that
$$
\bigl|\sigma(x,0)-\sigma(\eta,0)\bigr|^{2}
= \frac{4\,|x-\eta|^{2}}{\bigl(1+x^{2}\bigr)\bigl(1+\eta^{2}\bigr)},
$$
we easily estimate the right-hand side of \eqref{eq:314} by $I(h/\alpha)$ (see \eqref{eq:216}) for appropriate constants and with $G$ replaced by $\psi$. Applying Lemma \ref{lm:2.3} and then using \eqref{eq:310}, \eqref{eq:313}, we conclude the proof.

\end{proof}

\begin{proof}[Proof of Theorem~\ref{thm:1.1}]
Fix $\nu>0$, $\varepsilon_{0}>0$, set $\nu_{0}=\nu+\varepsilon_{0}$, $\rho=\rho'=1-\varepsilon_{0}$, and define $\Omega$, $\Omega'$ as in \eqref{eq:302}, \eqref{eq:304}. By \eqref{eq:305}, the first bound in \eqref{eq:208} (with $N=0$), and Propositions ~\ref{prop:3.1} and ~\ref{prop:3.2}, we have
\begin{equation*}\tag{3.15}\label{eq:315}
\log |F(\omega)| \le \|A(\omega)^{4}\|_{\operatorname{tr}(\mathcal{X})}
\le h^{-4\,(\nu-\beta_{BD})+O(\varepsilon_{0})},\qquad \omega\in\Omega.
\end{equation*}

Next, we need a lower bound on $F$ at a point of $\Omega'$. If $\varepsilon_{0}$ is sufficiently small, then similarly to \cite[(4.7)]{Dya19b}, we may use \eqref{eq:303} to see that \begin{equation*}\tag{3.16}\label{eq:316}
\bigl\|A(\omega_{0})^{4}\bigr\|_{\mathcal{X}\to\mathcal{X}} \le \tfrac{1}{2},
\qquad \omega_{0}:=1+\tfrac{i h}{3}\in\Omega'.
\end{equation*}

\noindent Moreover,
\[
\bigl(I-A(\omega_{0})^{4}\bigr)^{-1}
= I + A(\omega_{0})^{4}\bigl(I-A(\omega_{0})^{4}\bigr)^{-1},
\]
and, using \eqref{eq:316} and the second inequality in \eqref{eq:315}, we get
\[
\bigl\|A(\omega_{0})^{4}\bigl(I-A(\omega_{0})^{4}\bigr)^{-1}\bigr\|_{\operatorname{tr}(\mathcal{X})}
\le C\,h^{-4\,(\nu-\beta_{BD})+O(\varepsilon_{0})}.
\]

\noindent Thus,
\[
-\log\bigl|F(\omega_{0})\bigr|
= \log\Bigl|\operatorname{det}\!\Bigl(I + A(\omega_{0})^{4}\bigl(I-A(\omega_{0})^{4}\bigr)^{-1}\Bigr)\Bigr|
\le C\,h^{-4\,(\nu-\beta_{BD})+O(\varepsilon_{0})},
\]
providing our lower bound.

As in \cite[proof of Theorem 1]{Dya19b}, we conclude by applying the Jensen inequality as outlined in \cite[proof of Theorem 2]{DD13} to see that the number of zeros of $F$ in $\Omega'$ is bounded by
$C\,h^{-4\,(\nu-\beta_{BD})+O(\varepsilon_{0})}$. As explained in the beginning of this section, this implies the bound
\[
\mathcal N(R,\nu)\le C\,R^{4\,(\nu-\beta_{BD})+O(\varepsilon_{0})},
\]
and taking $\varepsilon_{0}$ sufficiently small concludes the proof.

\end{proof}
\subsection{Improved resolvent bound in the improved spectral gap.} With almost all of the work already done, the proof of Theorem~\ref{thm:1.2} is now simple.

\begin{proof}[Proof of Theorem~\ref{thm:1.2}]
As in \cite[proof of Theorem 2]{Dya19b}, it suffices to prove the bound
\begin{equation*}
\bigl\|P_{h}(\omega)^{-1}\bigr\|_{\mathcal{Y}\to\mathcal{X}}
\le C\,h^{-1-c(\nu,\delta)-\varepsilon},
\qquad
\omega \in \Omega := [1-2h,1+2h] + ih[-\nu,1].
\end{equation*}

\noindent (Note that this $\Omega$ is different than the one in \eqref{eq:302}.) Indeed, by \eqref{eq:204}, this implies an $H_{h}^{s-1}(M)\to H_{h}^{s}(M)$ estimate on $\psi R_{h}(\omega)\psi$ for any $\psi\in C_{c}^{\infty}(M)$, and using an elliptic parametrix for $-h^{2}\Delta_{M}-\tfrac{h^{2}}{4}-\omega^{2}$ (see, for instance, \cite[Proposition E.32]{DZw19}), we convert this to an $L^{2}\to L^{2}$ estimate, which implies \eqref{eq:102} after recalling \eqref{eq:201}.

Fix $\varepsilon_0 > 0$ and set

$$
\rho=1-\varepsilon_{0},\qquad
\rho'=\frac{2\bigl(1-\delta-2\beta_{BD}\bigr)+\sqrt{\varepsilon_{0}}}{2\bigl(1-\delta-2\nu\bigr)} .
$$

\noindent Note that $\rho'<1$ for sufficiently small $\varepsilon_{0}$ (depending on $\nu$). Since $s_{0}(A)=\|A\|$ for any compact operator, we use \eqref{eq:307} to see that
$$
\|A(\omega)\|_{\mathcal{X}\to\mathcal{X}}^{4}
\le C\,\|J(\omega)\|_{L^{2}\to L^{2}}^{4}\,\operatorname{tr}\!\bigl((\tilde{A}^{*}\tilde{A})^{2}\bigr)
+O(h^{\infty}),\qquad \omega\in\Omega.
$$

\noindent By \eqref{eq:208}, Proposition ~\ref{prop:3.2}, and our choices of $\rho,\rho'$ above, this gives
$$
\|A(\omega)\|_{\mathcal{X}\to\mathcal{X}}^{4}
\le C\,h^{\sqrt{\varepsilon_{0}}+O(\varepsilon_{0})}
\le \tfrac{1}{2},\qquad \omega\in\Omega,
$$
for small enough $\varepsilon_{0}$ and small $h$. Then we also have
$$
\|(I-A(\omega))^{-1}\|_{\mathcal{X}\to\mathcal{X}} \le C,\qquad \omega\in\Omega,
$$
for small $h$. Since $P_{h}(\omega)^{-1}=(I-A(\omega))^{-1}Z(\omega):\mathcal{Y}\to\mathcal{X}$ (see \eqref{eq:205}), and
$$
\|Z(\omega)\|_{\mathcal{Y}\to\mathcal{X}}
\le C\,h^{-1-(\rho+\rho')(\nu+\varepsilon_{0})},\qquad \omega\in\Omega
,
$$
(see \eqref{eq:206}) we obtain
$$
\|P_{h}(\omega)^{-1}\|_{\mathcal{Y}\to\mathcal{X}}
\le C\,h^{-1-c(\nu,\delta)+O(\sqrt{\varepsilon_{0}})},\qquad \omega\in\Omega.
$$

\noindent Since $\varepsilon_{0}$ is arbitrary, the proof is complete.
\end{proof}

\section{PROOF OF THEOREM 1.3}

Here we shall apply an argument analogous to that in Section 3.1 but in the context of quantum open baker's maps, to prove Theorem~\ref{thm:1.3}. Throughout this section, for any $N\in\mathbb{N}$, we consider the abelian group
$$
\mathbb{Z}_{N}=\mathbb{Z}/N\mathbb{Z}\simeq\{0,1,\ldots,N-1\},
$$
and define the Hilbert space $\ell_{N}^{2}$ as the space of functions $u:\mathbb{Z}_{N}\to\mathbb{C}$ with norm
$$
\|u\|_{\ell_{N}^{2}}^{2}=\sum_{j=0}^{N-1}|u(j)|^{2}.
$$

We recall from the introduction that a quantum open baker's map, $B_{N}$, is determined by a triple $(M,\mathcal A,\chi)$, where $M\in\mathbb{N}$ is the base, $\mathcal A\subset \mathbb{Z}_{M} = \{0,1,\ldots,M-1\}$ is the alphabet, and $\chi\in C_{c}^{\infty}((0,1);[0,1])$ is a cutoff function. Then for $N=M^{k}$, $k\in\mathbb{N}$,
$
B_{N}:\ell_{N}^{2}\to \ell_{N}^{2}
$
is the sequence of operators given by

\begin{equation*}\tag{4.1}\label{eq:401}
B_{N}:=\mathcal{F}_{N}^{*}
\begin{pmatrix}
\chi_{N/M}\,\mathcal{F}_{N/M}\,\chi_{N/M} & & & \\
& \ddots & & \\
& & \chi_{N/M}\,\mathcal{F}_{N/M}\,\chi_{N/M}
\end{pmatrix}
\, I_{\mathcal A,M}.
\end{equation*}

\noindent where $\mathcal{F}_{N}:\ell_{N}^{2}\to \ell_{N}^{2}$ is the unitary Fourier transform,
$$
(\mathcal{F}_{N}u)(j)=\frac{1}{\sqrt{N}}\sum_{l=0}^{N-1}\exp\!\left(-\frac{2\pi i\,jl}{N}\right)u(l),
$$
we let $\chi_{N/M}\in \ell_{N/M}^{2}$ be the discretization,

$$
\chi_{N/M}(j)=\chi\!\left(\frac{Mj}{N}\right), \qquad j\in\left\{0,1,\ldots,\frac{N}{M}-1\right\},
$$
(thus, in \eqref{eq:401}, $\chi_{N/M}$ is treated as a multiplication operator on $\ell_{N/M}^{2}$), and $I_{\mathcal A,M}$ is the diagonal matrix whose $j$th entry is $1$ if $\lfloor \frac{Mj}{N}\rfloor \in \mathcal A$, and $0$ otherwise. Notice that $ \operatorname{Sp}(B_{N})\subset\{\,|\lambda|\le 1\,\}$
since $\|B_{N}\|_{\ell_{N}^{2}\to \ell_{N}^{2}}\le 1$, as follows directly from \eqref{eq:401}.

We next recall that
\begin{equation*}\tag{4.2}\label{eq:402}
\delta:=\frac{\log |\mathcal A|}{\log M}\in(0,1)
\end{equation*}
is the dimension of the limiting Cantor set
$$
C_{\infty}:=\bigcap_{k}\;\bigcup_{j\in C_{k}}\left[\frac{j}{M^{k}},\,\frac{j+1}{M^{k}}\right],
$$
where

\begin{equation*}\tag{4.3}\label{eq:403}
C_{k}=C_{k}(M,\mathcal A):=\left\{\sum_{j=0}^{k-1} a_{j} M^{j}\;:\; a_{0},\ldots,a_{k-1}\in \mathcal A\right\}\subset \mathbb{Z}_N.
\end{equation*}

\noindent The Cantor set $C_{\infty}$ plays the role of the limit set $\Lambda_{\Gamma}$ in this context.

As mentioned above, the proof of Theorem~\ref{thm:1.3} follows the approach of Section 3.1 (see also \cite[Section 4.2]{DJ17}), and we begin with an approximate inverse identity for the operator $B_{N}-\lambda$ which is analogous to \eqref{eq:205}.

Fix $\nu_{0}>0$ and set
\begin{equation*}\tag{4.4}\label{eq:404}
\Omega:=\left\{\,M^{-\nu_{0}}<|\lambda|<3\,\right\}\subset\mathbb{C}.
\end{equation*}

\noindent Let $\rho\in(0,1)$ and define
\begin{equation*}\tag{4.5}\label{eq:405}
X=X_{\rho}:=\bigcup\left\{\,C_{k}+m \;:\; m\in\mathbb{Z},\ |m|\le 2\,N^{1-\rho}\,\right\}\subset \mathbb{Z}_{N},
\end{equation*}
where addition is carried out in the group $\mathbb{Z}_{N}$.

Dyatlov--Jin, \cite[Lemma~4.1]{DJ17}, prove the following approximate inverse identity:
\begin{equation*}\tag{4.6}\label{eq:406}
I \;=\; Z(\lambda)\,\bigl(B_{N}-\lambda\bigr)\;+\;B(\lambda), \qquad \lambda\in\Omega,
\end{equation*}
for families of operators on $\ell_{N}^{2}$, holomorphic in $\Omega$, satisfying
\begin{equation*}
\|Z(\lambda)\|_{\ell_{N}^{2}\to \ell_{N}^{2}} \;\le\; C\,N^{2\rho \nu_{0}}, \qquad \lambda\in\Omega,
\end{equation*}
and
\begin{equation*}\tag{4.7}\label{eq:407}
B(\lambda) \;:=\; J(\lambda)\, \mathbf{1}_{X}\,\mathcal{F}_{N}^{*}\,\mathbf{1}_{X}\,\mathcal{F}_{N}\;+\;\varepsilon(\lambda), \qquad \lambda\in\Omega,
\end{equation*}

\noindent where (here $\bar{k}:=\lceil \rho k\rceil$)

\begin{equation*}\tag{4.8}\label{eq:408}
\|J(\lambda)\|_{\ell_{N}^{2}\to \ell_{N}^{2}} \le |\lambda|^{-\bar{k}}, 
\qquad 
\|\varepsilon(\lambda)\|_{\ell_{N}^{2}\to \ell_{N}^{2}} = O\!\left(N^{-\infty}\right),
\qquad \lambda\in\Omega.
\end{equation*}

Analogously to \eqref{eq:301}, we now define
\begin{equation*}
F(\lambda):=\det\!\bigl(I-B(\lambda)^{4}\bigr), \qquad \lambda\in\Omega.
\end{equation*}
Using \eqref{eq:406} it is easy to see 
\begin{equation*}\
F(\lambda)
=\det\!\bigl(I+B(\lambda)+B(\lambda)^{2}+B(\lambda)^{3}\bigr)\;
\det Z(\lambda)\;
\det\!\bigl(B_{N}-\lambda\bigr), \qquad \lambda\in\Omega.
\end{equation*}

\noindent so that 
\begin{equation*}
\operatorname{Sp}(B_{N})\cap\Omega \;\subset\; \{\lambda\in\Omega : F(\lambda)=0\},
\end{equation*}
both sets repeating elements with multiplicities. Therefore, to prove Theorem~\ref{thm:1.3} it suffices to give the same bound on the number of zeros of $F$. We will again apply a Jensen inequality, and to do so we need to bound $F$ from above in $\Omega$, and from below at some point.
To get an upper bound, we apply \cite[(B.5.19)]{DZw19},
\begin{equation} \tag{4.9}\label{eq:409}
\log |F(\lambda)| \leq \bigl\|B(\lambda)^{4}\bigr\|_{\operatorname{tr}}
\end{equation}
and then, in exactly the same way as in the proof of Proposition \ref{prop:3.1}, we obtain
\begin{equation} \tag{4.10}\label{eq:410}
\bigl\|B(\lambda)^{4}\bigr\|_{\operatorname{tr}}
\leq \sum_{j=0}^{N-1} s_{j}\bigl(B(\lambda)\bigr)^{4}
\leq C\,\|J(\lambda)\|_{\ell_{N}^{2} \rightarrow \ell_{N}^{2}}^{4}\,
\operatorname{tr}\!\bigl((T^{*} T)^{2}\bigr),
\end{equation}
where
\begin{equation} \tag{4.11}\label{eq:411}
T:=\mathbf{1}_{X}\,\mathcal{F}_{N}\,\mathbf{1}_{X}.
\end{equation}

\noindent We have

\begin{equation}\tag{4.12}\label{eq:412}
\operatorname{tr}\!\bigl((T^{*} T)^{2}\bigr)
=\frac{1}{N}\sum_{j,\ell=0}^{N-1} \mathbf{1}_{X}(j)\,\mathbf{1}_{X}(\ell)\,
\bigl|\mathcal{F}_{N}(\mathbf{1}_{X})(\ell-j)\bigr|^{2},
\end{equation}

\noindent which can be seen by summing the diagonal entries of the matrix $(T^{*} T)^{2}$. It is interesting to compare this quantity to its analogue, \eqref{eq:216}, in the setting of hyperbolic surfaces.

Next we reduce $\operatorname{tr}((T^{*} T)^{2})$ to a quantity which is easier to estimate:

\begin{lm}\label{lm:4.1}
With $C_{k}$ as defined in \eqref{eq:403}, set
\begin{equation*}
t_{k}:=\frac{1}{N} \sum_{j, \ell=0}^{N-1} \textbf{1}_{C_{k}}(j)\, \textbf{1}_{C_{k}}(\ell)\,\left|\mathcal F_{N}\!\left(\textbf{1}_{C_{k}}\right)(\ell-j)\right|^{2}. \tag{4.13}\label{eq:4.13}
\end{equation*}
Then with $T$ as in \eqref{eq:411}, we have
\begin{equation*}
\operatorname{tr}\left(\left(T^{*} T\right)^{2}\right) \leq 16 N^{4(1-\rho)} t_{k}.
\end{equation*}
\end{lm}

\begin{proof}
Using $0\le \textbf{1}_{X}\le \sum_{|r|\le 2N^{1-\rho}} \textbf{1}_{C_{k}+r}$, we obtain from \eqref{eq:412} and \eqref{eq:405}
\begin{align*}
\operatorname{tr}\!\bigl((T^{*}T)^{2}\bigr)
&\le \sum_{|r|,|s|\le 2N^{1-\rho}} \frac{1}{N}
   \sum_{j,\ell=0}^{N-1}
   \textbf{1}_{C_{k}+r}(j)\,\textbf{1}_{C_{k}+s}(\ell)\,
   \bigl|\mathcal{F}_{N}\!\bigl(\textbf{1}_{X}\bigr)(\ell-j)\bigr|^{2} \\[0.25em]
&=   \sum_{|r|,|s|\le 2N^{1-\rho}} \frac{1}{N}
     \sum_{j,\ell=0}^{N-1}
     \textbf{1}_{C_{k}}(j)\,\textbf{1}_{C_{k}}(\ell)\,
     \bigl|\mathcal{F}_{N}\!\bigl(\textbf{1}_{X}\bigr)(\ell-j+s-r)\bigr|^{2}.
\tag{4.14}\label{eq:4.14}
\end{align*}

Now,
\[
\begin{aligned}
\bigl|\mathcal{F}_{N}\!\bigl(\textbf{1}_{X}\bigr)(\ell-j+s-r)\bigr|^{2}
&= \frac{1}{N}\sum_{n,m=0}^{N-1}
    \exp\!\Bigl(-\tfrac{2\pi i(\ell-j+s-r)n}{N}\Bigr)\,\textbf{1}_{X}(n)\,
    \exp\!\Bigl(\tfrac{2\pi i(\ell-j+s-r)m}{N}\Bigr)\,\textbf{1}_{X}(m) \\[0.15em]
&= \frac{1}{N}\sum_{n,m=0}^{N-1}
    \exp\!\Bigl(\tfrac{2\pi i(s-r)(m-n)}{N}\Bigr)\,
    \exp\!\Bigl(\tfrac{2\pi i(m-n)\ell}{N}\Bigr)\,
    \exp\!\Bigl(-\tfrac{2\pi i(m-n)j}{N}\Bigr)\,
    \textbf{1}_{X}(n)\,\textbf{1}_{X}(m).
\end{aligned}
\]

Rearranging the sums, using \eqref{eq:4.14}, and repeating the same estimate yields
\[
\begin{aligned}
\operatorname{tr}\!\bigl((T^{*}T)^{2}\bigr)
&\le \sum_{|r|,|s|\le 2N^{1-\rho}} \frac{1}{N}
     \sum_{n,m=0}^{N-1}
     \exp\!\Bigl(\tfrac{2\pi i(s-r)(m-n)}{N}\Bigr)\,\textbf{1}_{X}(n)\,\textbf{1}_{X}(m)\,
     \bigl|\mathcal{F}_{N}\!\bigl(\textbf{1}_{C_{k}}\bigr)(m-n)\bigr|^{2} \\[0.15em]
&\le 4\,N^{2(1-\rho)}\cdot \frac{1}{N}
     \sum_{n,m=0}^{N-1} \textbf{1}_{X}(n)\,\textbf{1}_{X}(m)\,
     \bigl|\mathcal{F}_{N}\!\bigl(\textbf{1}_{C_{k}}\bigr)(m-n)\bigr|^{2} \\[0.15em]
&\le 4\,N^{2(1-\rho)} \sum_{|t|,|s|\le 2N^{1-\rho}} \frac{1}{N}
     \sum_{n,m=0}^{N-1} \textbf{1}_{C_{k}+t}(n)\,\textbf{1}_{C_{k}+s}(m)\,
     \bigl|\mathcal{F}_{N}\!\bigl(\textbf{1}_{C_{k}}\bigr)(m-n)\bigr|^{2} \\[0.15em]
&\le 4\,N^{2(1-\rho)} \sum_{|t|,|s|\le 2N^{1-\rho}} \frac{1}{N}
     \sum_{n,m=0}^{N-1} \textbf{1}_{C_{k}}(n)\,\textbf{1}_{C_{k}}(m)\,
     \bigl|\mathcal{F}_{N}\!\bigl(\textbf{1}_{C_{k}}\bigr)(m-n+s-t)\bigr|^{2}.
\end{aligned}
\]

\[
\begin{aligned}
&= 4\,N^{2(1-\rho)}
   \sum_{|t|,|s|\le 2N^{1-\rho}}
   \frac{1}{N}\sum_{j,\ell=0}^{N-1}
   \exp\!\Bigl(\tfrac{2\pi i\, (s-t)\,(\ell-j)}{N}\Bigr)\,
   \textbf{1}_{C_{k}}(j)\,\textbf{1}_{C_{k}}(\ell)\,
   \bigl|\mathcal{F}_{N}\!\bigl(\textbf{1}_{C_{k}}\bigr)(\ell-j)\bigr|^{2} \\[0.15em]
&\le 16\,N^{4(1-\rho)}\,t_{k}.
\end{aligned}
\]
\end{proof}

We will prove two upper bounds on $t_{k}$, Propositions ~\ref{prop:4.3} and ~\ref{prop:4.4}, corresponding to the improved spectral gaps beating $\frac{1}{2}-\delta$ (\cite[Theorem 1]{DJ17}) and in terms of additive energy (\cite[§3.4]{DJ17}), respectively. In \cite{DJ17}, spectral gaps are proved via the fractal uncertainty principle, which in this context is an estimate on the quantity
$r_{k}:=\left\|\,\textbf{1}_{C_{k}}\;\mathcal{F}_{N}\;\textbf{1}_{C_{k}}\,\right\|_{\ell_{N}^{2} \rightarrow \ell_{N}^{2}}$ (see also \cite[§4.1]{Dya19a}). A crucial initial observation in \cite[§3.1]{DJ17} is that $r_{k}$ is submultiplicative; likewise, we will show next that $t_{k}$ is submultiplicative.

We first make an observation that will be used in the proof. Let $k=k_{1}+k_{2}$ and $N_{i}=M^{k_{i}}$ for $i=1,2$. If $\ell \in\{0,\ldots,N-1\}$, we can write $\ell=N_{2}a+e$ for some unique $a \in\{0,\ldots,N_{1}-1\}$ and $e \in\{0,\ldots,N_{2}-1\}$. Moreover, since $C_{k}=N_{2}C_{k_{1}}+C_{k_{2}}$, we have $\textbf{1}_{C_{k}}(\ell)=\textbf{1}_{C_{k_{1}}}(a)\,\textbf{1}_{C_{k_{2}}}(e)$. Thus, for any $f=\{f_{\ell}\}_{\ell=0}^{N-1}$, we have
\begin{equation*}
\sum_{\ell=0}^{N-1} \textbf{1}_{C_{k}}(\ell)\, f_{\ell}
=\sum_{a=0}^{N_{1}-1} \sum_{e=0}^{N_{2}-1} \textbf{1}_{C_{k_{1}}}(a)\,\textbf{1}_{C_{k_{2}}}(e)\, f_{N_{2} a+e}.
\tag{4.15}\label{eq:4.15}
\end{equation*}

We now prove the submultiplicative property of $t_{k}$:

\begin{lm} \label{lm:4.2}
    For $k_1, k_2 \in \N$, we have: 

    $$t_{k_1 + k_2} \le t_{k_1}t_{k_2}$$
\end{lm}

\begin{proof}
We may write
\begin{equation}\tag{4.16}\label{eq:4.16}
t_{k}
=\frac{1}{N^{2}}\sum_{j,\ell,m,n=0}^{N-1}
\exp\!\Bigl(\tfrac{2\pi i\,(\ell-j)(m-n)}{N}\Bigr)\,
\textbf{1}_{C_{k}}(j)\,\textbf{1}_{C_{k}}(\ell)\,
\textbf{1}_{C_{k}}(n)\,\textbf{1}_{C_{k}}(m).
\end{equation}

\noindent We now change variables as discussed prior to the proof, writing
$$
\begin{aligned}
&\ell=N_{2}a+e,\qquad &&j=N_{2}b+f,\qquad &&a,b,c,d\in C_{k_{1}},\\
&m=N_{1}g+c,\qquad &&n=N_{1}h+d,\qquad &&e,f,g,h\in C_{k_{2}}.
\end{aligned}
$$

\noindent Noting that the exponential term becomes
$$
\exp\!\Bigl(\tfrac{2\pi i\,(\ell-j)(m-n)}{N}\Bigr)
=\exp\!\Bigl(\tfrac{2\pi i\,(a-b)(c-d)}{N_{1}}\Bigr)\,
 \exp\!\Bigl(\tfrac{2\pi i\,(e-f)(g-h)}{N_{2}}\Bigr)\,
 \exp\!\Bigl(\tfrac{2\pi i\,(e-f)(c-d)}{N}\Bigr),
$$
and then using \eqref{eq:4.15} (repeatedly), we compute
$$
\begin{aligned}
t_{k}
&=\frac{1}{N^{2}}
  \sum_{a,b,c,d=0}^{N_{1}-1}\;
  \sum_{e,f,g,h=0}^{N_{2}-1}
  \exp\!\Bigl(\tfrac{2\pi i\,(a-b)(c-d)}{N_{1}}\Bigr)\,
  \exp\!\Bigl(\tfrac{2\pi i\,(e-f)(g-h)}{N_{2}}\Bigr)\,
  \exp\!\Bigl(\tfrac{2\pi i\,(e-f)(c-d)}{N}\Bigr)\\
&\qquad\cdot
  \textbf{1}_{C_{k_{1}}}(a)\,\textbf{1}_{C_{k_{1}}}(b)\,\textbf{1}_{C_{k_{1}}}(c)\,\textbf{1}_{C_{k_{1}}}(d)\,
  \textbf{1}_{C_{k_{2}}}(e)\,\textbf{1}_{C_{k_{2}}}(f)\,\textbf{1}_{C_{k_{2}}}(g)\,\textbf{1}_{C_{k_{2}}}(h).
\end{aligned}
$$
But $$
\bigl|\mathcal{F}_{N_{1}}\!\bigl(\textbf{1}_{C_{k_{1}}}\bigr)(c-d)\bigr|^{2}
=\frac{1}{N_{1}}\sum_{a,b=0}^{N_{1}-1}
\exp\!\Bigl(\tfrac{2\pi i\,(a-b)(c-d)}{N_{1}}\Bigr)\,
\textbf{1}_{C_{k_{1}}}(a)\,\textbf{1}_{C_{k_{1}}}(b),
$$
and
$$
\bigl|\mathcal{F}_{N_{2}}\!\bigl(\textbf{1}_{C_{k_{2}}}\bigr)(e-f)\bigr|^{2}
=\frac{1}{N_{2}}\sum_{g,h=0}^{N_{2}-1}
\exp\!\Bigl(\tfrac{2\pi i\,(e-f)(g-h)}{N_{2}}\Bigr)\,
\textbf{1}_{C_{k_{2}}}(g)\,\textbf{1}_{C_{k_{2}}}(h),
$$
so we have, 
$$
\begin{aligned}
t_{k}
&= \frac{1}{N}
   \sum_{c,d=0}^{N_{1}-1} \sum_{e,f=0}^{N_{2}-1}
   \exp\!\Bigl(\tfrac{2\pi i\,(e-f)(c-d)}{N}\Bigr)\,
   \bigl|\mathcal{F}_{N_{1}}\!\bigl(\textbf{1}_{C_{k_{1}}}\bigr)(c-d)\bigr|^{2}\,
   \bigl|\mathcal{F}_{N_{2}}\!\bigl(\textbf{1}_{C_{k_{2}}}\bigr)(e-f)\bigr|^{2} \\
&\qquad\cdot
   \textbf{1}_{C_{k_{1}}}(c)\,\textbf{1}_{C_{k_{1}}}(d)\,
   \textbf{1}_{C_{k_{2}}}(e)\,\textbf{1}_{C_{k_{2}}}(f) \\[0.25em]
&\le \frac{1}{N}
   \sum_{c,d=0}^{N_{1}-1} \sum_{e,f=0}^{N_{2}-1}
   \bigl|\mathcal{F}_{N_{1}}\!\bigl(\textbf{1}_{C_{k_{1}}}\bigr)(c-d)\bigr|^{2}\,
   \bigl|\mathcal{F}_{N_{2}}\!\bigl(\textbf{1}_{C_{k_{2}}}\bigr)(e-f)\bigr|^{2}\,
   \textbf{1}_{C_{k_{1}}}(c)\,\textbf{1}_{C_{k_{1}}}(d)\,
   \textbf{1}_{C_{k_{2}}}(e)\,\textbf{1}_{C_{k_{2}}}(f) \\[0.25em]
&= t_{k_{1}}\,t_{k_{2}}.
\end{aligned}
$$
\end{proof}

We now use Lemma~\ref{lm:4.2} to prove the following estimate on $t_{k}$, which will allow us to prove the portion of our fractal Weyl bound which matches the spectral gap beating $\frac{1}{2}-\delta$:

\begin{prop}\label{prop:4.3}
    
 There exists $\beta>\frac{1}{2}-\delta$ such that for any $\varepsilon>0$, there exists $C_{\varepsilon}>0$ so that
\begin{equation}\tag{4.17}\label{eq:4.17}
t_{k} \leqslant C_{\varepsilon} N^{-4 \beta+\varepsilon}.
\end{equation}

\end{prop}

\begin{proof}
    
Since $\log t_{k}$ is subadditive by Lemma~\ref{lm:4.2}, Feket\'{e}'s Lemma shows that there exists $\beta$ for which
\begin{equation}\tag{4.18}\label{eq:4.18}
4 \beta=-\lim_{k \to \infty} \frac{\log t_{k}}{k \log M}=-\inf_{k} \frac{\log t_{k}}{k \log M}.
\end{equation}

\noindent This implies \eqref{eq:4.17} and it remains to show that $\beta>\frac{1}{2}-\delta$; by \eqref{eq:4.18}, it suffices to prove that $t_{k}<N^{4 \delta-2}$ for \textit{some} $k$, and we will follow \cite[Lemma~4.8]{Dya19a} (see also \cite[Lemma~3.4]{DJ17}) to show that this in fact holds for all $k \geq 2$.

Indeed, we first note that by \eqref{eq:4.16}, we easily obtain
\begin{equation}\tag{4.19}\label{eq:4.19}
t_{k} \leqslant \frac{\left|C_{k}\right|^4}{N^{2}}=N^{4\left(\delta-\frac{1}{2}\right)},
\end{equation}

\noindent since $\left|C_{k}\right|=|\mathcal A|^{k}=N^{\delta}$ by \eqref{eq:403}, \eqref{eq:402}. From \eqref{eq:4.16} again, we see that equality can hold in \eqref{eq:4.19} only if the exponential factor is $\equiv 1$ for all $j,\ell,m,n \in C_{k}$, that is, only if $(\ell-j)(m-n) \in N \mathbb{Z}$ for all $j,\ell,m,n \in C_{k}$. But when $k \geq 2$, we can find $\ell,j \in C_{k}$ such that
$$
0<|\ell-j|<M \leqslant \sqrt{N}
$$
(using here that $\delta>0$), so $(\ell-j)^{2} \notin N \mathbb{Z}$ and \eqref{eq:4.19} must be a strict inequality whenever $k \geq 2$.
\end{proof}

Next we prove another bound on $t_{k}$, this one in terms of an additive energy estimate given by \cite[§3.4]{DJ17}. To state it, recall that the additive energy of a set $X \subset \mathbb{Z}_{N}$ is defined by
\[
E_{\mathcal A}(X):=\left|\left\{(a,b,c,d)\in X^{4}: a+b\equiv c+d \pmod N\right\}\right|.
\]

In \cite[Corollary 3.9 and Lemma 3.10]{DJ17}, Dyatlov–Jin prove that there exists $\gamma_{\mathcal A}>0$ such that for any $\varepsilon>0$, we have for some $C_{\varepsilon}>0$ and all $k$,
\begin{equation}\tag{4.20}\label{eq:4.20}
E_{\mathcal A}\!\left(C_{k}\right) \leq C_{\varepsilon}\, N^{3\delta - \gamma_{\mathcal A} + \varepsilon}.
\end{equation}

\noindent Moreover, $\gamma_{\mathcal A}$ can be explicitly computed. By first proving a fractal uncertainty principle, this estimate is used in \cite{DJ17} to prove a spectral
gap of size $$
\beta_{E}:=\frac{3}{4}\left(\frac{1}{2}-\delta\right)+\frac{\gamma_{\mathcal A}}{8}.
$$

We now prove our second bound on $t_{k}$ :

\begin{prop}\label{prop:4.4}
For any $\varepsilon>0$, there exists $C_{\varepsilon}>0$ such that
\begin{equation*}
t_{k} \leq C_{\varepsilon}\, N^{-4 \beta_{E}+\varepsilon}.
\end{equation*}
\end{prop}

\begin{proof}
Schwarz inequality gives
\begin{equation}\tag{4.21}\label{eq:4.21}
\sum_{\ell=0}^{N-1} \textbf{1}_{C_{k}}(\ell)\,\bigl|\mathcal{F}_{N}\bigl(\textbf{1}_{C_{k}}\bigr)(\ell-j)\bigr|^{2}
\ \leq\ \lvert C_{k}\rvert^{1/2}\,\Bigl\|\mathcal{F}_{N}\bigl(\textbf{1}_{C_{k}}\bigr)\Bigr\|_{\ell_{N}^{4}}^{2}.
\end{equation}

\noindent But we compute
\begin{equation*}
\Bigl\|\mathcal{F}_{N}\bigl(\textbf{1}_{C_{k}}\bigr)\Bigr\|_{\ell_{N}^{4}}^{4}
= \frac{1}{N^{2}}\sum_{\ell=0}^{N-1}\ \sum_{a,b,c,d\in C_{k}}
\exp\!\left(\frac{2\pi i\,(a+b-c-d)\,\ell}{N}\right)
= \frac{E_{\mathcal A}\!\left(C_{k}\right)}{N}.
\end{equation*}
Thus from \eqref{eq:4.21} and the definition of $t_{k}$ in \eqref{eq:4.13}, we have

$$
t_{k} \leqslant N^{-3/2}\,\lvert C_{k}\rvert^{3/2}\, E_{\mathcal A}\!\left(C_{k}\right)^{1/2}.
$$

\noindent Since $\lvert C_{k}\rvert = N^{\delta}$, the claim follows from \eqref{eq:4.20}.
\end{proof}
We are now ready to begin the proof of Theorem~\ref{thm:1.3}.

\begin{proof}[Proof of Theorem~\ref{thm:1.3}]
Fix $\nu >0$, $\varepsilon_{0}>0$, set $\nu_{0}=\nu+\varepsilon_{0}$, $\rho=1-\varepsilon_{0}$, and define $\Omega$ as in \eqref{eq:404}. Using \eqref{eq:410}, the first bound in \eqref{eq:408}, and Lemma~\ref{lm:4.1}, we obtain
$$
\bigl\|B(\lambda)^{4}\bigr\|_{\operatorname{tr}} \leq C\, N^{4\nu + O(\varepsilon_{0})}\, t_{k}, \qquad \lambda \in \Omega .
$$
We now use \eqref{eq:409} and Proposition~\ref{prop:4.3} with $\varepsilon_{0}$ to get
\begin{equation}\tag{4.22}\label{eq:4.22}
\log\lvert F(\lambda)\rvert \ \le\ \bigl\|B(\lambda)^{4}\bigr\|_{\operatorname{tr}}
\ \le\ C\,N^{4(\nu-\beta)+O(\varepsilon_{0})}, \qquad \lambda\in\Omega .
\end{equation}

Next we use \eqref{eq:407} and \eqref{eq:408} to see that
\begin{equation}\tag{4.23}\label{eq:423}
\bigl\|B(2)^{4}\bigr\|_{\ell_{N}^{2}\to \ell_{N}^{2}} \ \le\ \tfrac{1}{2}
\end{equation}
for large enough $k$. Moreover,
\[
\bigl(I-B(2)^{4}\bigr)^{-1} \ =\ I\ +\ B(2)^{4}\bigl(I-B(2)^{4}\bigr)^{-1},
\]
and using \eqref{eq:423} and the second inequality in \eqref{eq:4.22}, we have
\[
\bigl\|B(2)^{4}\bigl(I-B(2)^{4}\bigr)^{-1}\bigr\|_{\operatorname{tr}}
\ \le\ C\,N^{4(\nu-\beta)+O(\varepsilon_{0})}.
\]

\noindent Thus,
\begin{equation}\tag{4.24}\label{eq:4.24}
-\log \lvert F(2)\rvert
\ =\ \log \Bigl\lvert \det\!\Bigl(I + B(2)^{4}\bigl(I-B(2)^{4}\bigr)^{-1}\Bigr)\Bigr\rvert
\ \le\ C\,N^{4(\nu-\beta)+O(\varepsilon_{0})}.
\end{equation}

Applying Jensen’s inequality as in \cite[Lemma~4.4]{DJ17} with \eqref{eq:4.22} and \eqref{eq:4.24}, we obtain the bound
\[
\mathcal N_{k}(\nu)\ \le\ C\,N^{4(\nu-\beta)+O(\varepsilon_{0})}
\qquad \text{as } k\to\infty .
\]

A similar proof using Proposition~\ref{prop:4.4} gives this estimate with $\beta$ replaced by $\beta_{E}$. Pairing these two bounds with \cite[Theorem~3]{DJ17} proves Theorem~\ref{thm:1.3}
\end{proof}

\bibliographystyle{alpha}

\nocite{Bor16}
\nocite{BD17}
\nocite{BD18}
\nocite{BGS11}
\nocite{DD13}
\nocite{Dya19a}
\nocite{Dya19b}
\nocite{DJ17}
\nocite{DW16}
\nocite{DZ16}
\nocite{DZw18}
\nocite{DZw19}
\nocite{GLZ04}
\nocite{MM87}
\nocite{Nau05}
\nocite{Nau14}
\nocite{Pat76}
\nocite{Sim79}
\nocite{Soa23}
\nocite{Sul79}
\nocite{Tao24}
\nocite{Vac24}
\nocite{Vac23}
\nocite{Vas13a}
\nocite{Vas13b}
\nocite{Zwo16}
\nocite{Zwo17}

\bibliography{references}

\noindent \emph{Email address}: travisdcunningham@gmail.com

\end{document}